\documentclass[11pt, reqno]{amsart}

\usepackage[top=3.75cm, bottom=3cm, left=3cm, right=3cm]{geometry}
\usepackage{amsmath}
\usepackage{amsthm}
\usepackage{amsfonts}
\usepackage{amssymb}
\usepackage{eucal}
\usepackage{bm}
\usepackage[all,cmtip]{xy}
\usepackage{hyperref}
\usepackage{color}
\usepackage{mathtools}
\usepackage{tabu}

\numberwithin{equation}{section}

\theoremstyle{plain}
\newtheorem{theorem}{Theorem}[section]
\newtheorem{lemma}[theorem]{Lemma}
\newtheorem{proposition}[theorem]{Proposition}
\newtheorem{corollary}[theorem]{Corollary}
\newtheorem{conjecture}[theorem]{Conjecture}

\theoremstyle{definition}
\newtheorem*{definition*}{Definition}
\newtheorem{definition}[theorem]{Definition}
\newtheorem{remark}[theorem]{Remark}

\newcommand{\Db}{\mathrm{D^b}}
\newcommand{\Perf}{\mathrm{Perf}}

\newcommand{\enh}{\mathrm{enh}}

\newcommand{\op}{\mathrm{op}}

\DeclareMathOperator{\G}{Gr}
\DeclareMathOperator{\Fl}{Fl}
\DeclareMathOperator{\Proj}{Proj}
\DeclareMathOperator{\Sym}{Sym}
\newcommand{\wtilde}{\widetilde}
\DeclareMathOperator{\Bl}{Bl}

\newcommand{\oQ}{\bar{Q}}

\newcommand{\Qv}{Q^{\vee}}

\newcommand{\bGv}{\mathbf{G}^{\vee}}

\newcommand{\deh}{\widehat\Delta}
\newcommand{\hg}{{\widehat{g}}}

\newcommand{\hX}{{\widehat{X}}}
\newcommand{\hY}{{\widehat{Y}}}
\newcommand{\hcE}{{\widehat{\cE}}}
\newcommand{\hPhi}{{\widehat{\Phi}}}
\DeclareMathOperator{\RCHom}{{\mathrm{R}\mathcal{H}\mathit{om}}}
\DeclareMathOperator{\PGL}{\mathrm{PGL}}
\DeclareMathOperator{\LGr}{\mathrm{LG}}

\DeclareMathOperator{\HH}{HH}

\DeclareMathOperator{\Ext}{Ext}

\DeclareMathOperator{\Spec}{Spec}
\DeclareMathOperator{\Hom}{Hom}

\DeclareMathOperator{\Tor}{Tor}
\DeclareMathOperator{\Aut}{Aut}

\newcommand{\rank}{\mathrm{rank}}

\DeclareMathOperator{\Cone}{Cone}
\DeclareMathOperator{\Pic}{Pic}

\newcommand{\gdim}{\mathrm{gdim}}
\newcommand{\num}{\mathrm{num}}
\newcommand{\CH}{\mathrm{CH}}

\newcommand{\ch}{\mathrm{ch}}
\newcommand{\van}{\mathrm{van}}
\newcommand{\opp}{\mathrm{op}}
\newcommand{\Tnsr}[1]{(-\otimes{#1})}
\newcommand{\Tr}{\mathrm{Tr}}
\newcommand{\coker}{\mathrm{coker}}

\newcommand{\cO}{\mathcal{O}}
\newcommand{\cA}{\mathcal{A}}
\newcommand{\cB}{\mathcal{B}}

\newcommand{\cE}{\mathcal{E}}
\newcommand{\cF}{\mathcal{F}}
\newcommand{\cG}{\mathcal{G}}
\newcommand{\cH}{\mathcal{H}}
\newcommand{\cI}{\mathcal{I}}
\newcommand{\cK}{\mathcal{K}}
\newcommand{\cL}{\mathcal{L}}
\newcommand{\cM}{{\mathcal M}}

\newcommand{\cS}{\mathcal{S}}
\newcommand{\cT}{\mathcal{T}}
\newcommand{\cU}{\mathcal{U}}
\newcommand{\cV}{\mathcal{V}}

\newcommand{\rF}{\mathrm{F}}

\newcommand{\rH}{\mathrm{H}}

\newcommand{\rK}{\mathrm{K}}
\newcommand{\rS}{\mathrm{S}}
\newcommand{\rL}{\mathrm{L}}

\newcommand{\rR}{\mathrm{R}}

\newcommand{\rT}{\mathrm{T}}

\newcommand{\bC}{\mathbf{k}}
\newcommand{\bG}{\mathbf{G}}
\newcommand{\bH}{\mathbf{H}}
\newcommand{\bZ}{\mathbf{Z}}
\newcommand{\bP}{\mathbf{P}}
\newcommand{\bQ}{\mathbf{Q}}
\newcommand{\bp}{\mathbf{p}}
\newcommand{\bq}{\mathbf{q}}

\newcommand{\sA}{\mathsf{A}}
\newcommand{\sY}{\mathsf{Y}}

\newcommand{\fA}{{\mathfrak{A}}}

\begin{document}

\title{Derived categories of Gushel--Mukai varieties}

\author{Alexander Kuznetsov}
\address{{\sloppy
\parbox{0.9\textwidth}{
Steklov Mathematical Institute of Russian Academy of Sciences,\\
8 Gubkin str., Moscow 119991 Russia
\\[5pt]
The Poncelet Laboratory, Independent University of Moscow
\hfill\\[5pt]
National Research University Higher School of Economics, Russian Federation
}\bigskip}}
\email{akuznet@mi.ras.ru \medskip}

\author{Alexander Perry}
\address{Department of Mathematics, Columbia University, New York, NY 10027, USA \medskip}
\email{aperry@math.columbia.edu}

\thanks{A.K. was partially supported by the Russian Academic Excellence Project ``5-100'', by RFBR grant 15-01-02164,
and by the Simons Foundation.
A.P. was partially supported by NSF GRFP grant DGE1144152 and NSF MSPRF grant DMS-1606460, 
and thanks the Laboratory of Algebraic Geometry NRU-HSE for its hospitality in December 2015, 
when part of this work was carried out.}

\date{\today}

\begin{abstract}
We study the derived categories of coherent sheaves on Gushel--Mukai varieties. 
In the derived category of such a variety, we isolate a special semiorthogonal 
component, which is a K3 or Enriques category according 
to whether the dimension of the variety is even or odd.
We analyze the basic properties of this category using Hochschild homology, 
Hochschild cohomology, and the Grothendieck group. 

We study the K3 category of a Gushel--Mukai fourfold in more detail. 
Namely, we show this category is equivalent to the derived category of 
a K3 surface for a certain codimension~1 family of rational Gushel--Mukai fourfolds, and to the K3 category 
of a birational cubic fourfold for a certain codimension 3 family. 
The first of these results verifies a special case of a duality 
conjecture which we formulate. We discuss our results in the context 
of the rationality problem for Gushel--Mukai varieties, which was one of the 
main motivations for this work.
\end{abstract}

\maketitle

\tableofcontents


\section{Introduction}
This paper studies the derived categories of coherent sheaves on smooth 
Gushel--Mukai varieties, with a special focus on the relation to birational 
geometry and the case of fourfolds. 

\subsection{Background}
For the purpose of this paper, we use the following definition. 

\begin{definition}
\label{definition-GM}
A \emph{Gushel--Mukai \textup{(}GM\textup{)} variety} is a smooth $n$-dimensional intersection
\begin{equation*}
X = \Cone(\G(2,5)) \cap \bP^{n+4} \cap Q,
\qquad 
2 \le n \le 6,
\end{equation*}
where $\Cone(\G(2,5)) \subset \bP^{10}$ is the cone over the Grassmannian $\G(2,5) \subset \bP^9$ 
in its Pl\"ucker embedding, $\bP^{n+4} \subset \bP^{10}$ is a linear subspace, and 
$Q \subset \bP^{n+4}$ is a quadric hypersurface.
\end{definition}

We note that a more general definition of GM varieties, which includes 
singular varieties and curves, is given in~\cite[Definition~2.1]{debarre-kuznetsov}. 
However, the definition there agrees with ours after imposing the condition that 
a GM variety is smooth of dimension at least $2$, see \cite[Proposition~2.28]{debarre-kuznetsov}.
The classification results of Gushel \cite{gushel1983fano} and Mukai~\cite{mukai}, 
generalized and simplified in~\cite[Theorem~2.16]{debarre-kuznetsov}, show that this class of varieties coincides with the class of all 
smooth Fano varieties of Picard number 1, coindex 3, and degree 10, 
together with Brill--Noether general polarized K3 surfaces of degree 10.

In the Fano--Iskovskikh--Mori--Mukai classification of Fano threefolds, 
GM threefolds occupy an intermediate position between complete intersections 
in weighted projective spaces and linear sections of homogeneous varieties, and possess a particularly rich birational geometry. 
The case of GM fourfolds is even more interesting, and was our original source of motivation. 
These fourfolds are similar to cubic fourfolds from several points of view: 
birational geometry, Hodge theory, and as we will see, derived categories.

In terms of birational geometry, both types of fourfolds are unirational and rational 
examples are known. On the other hand, a very general fourfold of either type is expected to be irrational,
but to date irrationality has not been shown for a single example.

At the level of Hodge theory, a fourfold of either type has middle cohomology of K3 type, i.e.~ $h^{0,4} = 0$ and $h^{1,3} = 1$. 
Moreover, there is a classification of Noether--Lefschetz loci where the ``non-special cohomology'' is isomorphic to (a Tate twist of) the primitive cohomology of a polarized K3 surface.
This is due to Hassett for cubics~\cite{hassett}, and Debarre--Iliev--Manivel~\cite{DIM4fold} for GM fourfolds. 

Finally, the first author studied the derived categories of cubic fourfolds in~\cite{kuznetsov2010derived}. 
For any cubic fourfold $X'$, a ``K3 category'' $\cA_{X'}$ is constructed as a semiorthogonal component 
of the derived category $\Db(X')$, and it is shown for many rational $X'$ that $\cA_{X'}$ is equivalent 
to the derived category of an actual K3 surface. 
Since their introduction, the categories $\cA_{X'}$ have attracted a great deal of attention, see for instance \cite{addington-thomas}, \cite{huybrechts-cubic}, \cite{macri}, \cite{calabrese-thomas}.

\subsection{GM categories}
We show in this paper that the parallel between GM and cubic fourfolds persists at the level 
of derived categories. 
In fact, for any GM variety $X$ 
--- not necessarily of dimension $4$~--- 
we define a semiorthogonal component $\cA_X$ 
of its derived category as the orthogonal to an exceptional sequence of vector bundles.  
Namely, projection from the vertex of~$\Cone(\G(2,5))$ gives a morphism $f\colon X \to \G(2,5)$, which corresponds to a rank $2$ bundle~$\cU_X$ on $X$. 
If $n = \dim X$, we show in Proposition~\ref{proposition-dbx} that there is a semiorthogonal decomposition 
\begin{equation*}
\Db(X) = \langle \cA_X, \cO_X, \cU_X^{\vee}, \cO_X(1), \cU_X^{\vee}(1), \dots,  \cO_X(n-3), \cU_X^{\vee}(n-3) \rangle.
\end{equation*}

The \emph{GM category} $\cA_X$ is the main object of study of this paper.
Its properties depend on the parity of the dimension $n$. 
For instance, we show that in terms of Serre functors, $\cA_X$ is a ``K3 category'' 
or ``Enriques category'' according to whether $n$ is even or odd (Proposition~\ref{proposition-serre-functor}).
We support the K3-Enriques analogy by showing that each GM category has a canonical involution such that 
the corresponding equivariant category is equivalent to a GM category of opposite parity (Proposition~\ref{proposition-GM-category-quotient}).

We also compute the Hochschild homology (Proposition~\ref{proposition-HH_*}), 
Hochschild cohomology (Corollary~\ref{corollary-even-GM-HC} and Proposition~\ref{proposition-odd-GM-HC}), 
and (in the very general case) the numerical Grothendieck group (Proposition~\ref{proposition-gg-AX} and Lemma~\ref{lemma-euler-form}) of GM categories. 
Our computation of Hochschild homology and Grothendieck groups is based on their 
additivity, while for Hochschild cohomology we rely on 
results on equivariant Hochschild cohomology from~\cite{equivariant-HC}.

We deduce from our computations structural properties of GM categories. 
Notably, we show that for any GM variety of odd dimension or
for a very general GM variety of even dimension greater than $2$, 
the category $\cA_X$ is not equivalent to the derived category 
of any variety (Proposition~\ref{proposition-AX-geometric}). 

\subsection{Conjectures on duality and rationality} 
We formulate two conjectures about GM categories. 
First we introduce a notion of ``generalized duality'' between GM varieties. 
The precise definition of this notion is somewhat involved (see \S\ref{subsection:duality}), 
but its salient features are as follows. 
Generalized duals have the same parity of dimension, and when 
they have the same dimension they are dual in the sense of \cite[Definition~3.26]{debarre-kuznetsov}. 
The space of generalized duals of $X$ is parameterized by the quotient of the projective space $\bP^5$ by a finite group. 
We also formulate a similar notion of ``generalized GM partners'', 
which reduces to \cite[Definition~3.22]{debarre-kuznetsov} when the 
varieties have the same dimension. 
We conjecture  
that generalized dual GM varieties and generalized GM partners have 
equivalent GM categories (Conjecture~\ref{conjecture:duality}). 

Our second conjecture concerns the rationality of GM fourfolds, and 
is directly inspired by an analogous conjecture for cubic fourfolds from~\cite{kuznetsov2010derived}. 
Namely, we conjecture that the GM category of a rational GM fourfold is equivalent to the derived category of a~K3 surface (Conjecture~\ref{conjecture-rational-fourfold}).
Together with Proposition~\ref{proposition-AX-geometric}, this conjecture implies  
that a very general GM fourfold is not rational. 

\subsection{Main results} 
Our first main result gives evidence for the above two conjectures. 
A GM variety as in Definition~\ref{definition-GM} is called ordinary if $\bP^{n+4}$ does not intersect the vertex of~$\Cone(\G(2,5))$. 

\begin{theorem}
\label{theorem-intro-K3}
Let $X$ be an ordinary GM fourfold containing a quintic del Pezzo surface. 
Then there is a $K3$ surface $Y$ such that $\cA_X \simeq \Db(Y)$. 
\end{theorem}

For a more precise statement, see Theorem~\ref{theorem-associated-K3}. 
The K3 surface $Y$ is in fact a GM surface which is generalized 
dual to $X$, and the GM fourfold $X$ is rational (Lemma~\ref{lemma-rational-GM}). 
Thus Theorem~\ref{theorem-intro-K3} verifies special cases 
of our duality and rationality conjectures. 
We note that GM fourfolds as in the theorem form a $23$-dimensional (codimension 1 in moduli) family.

By Theorem~\ref{theorem-intro-K3}, the categories $\cA_X$ of GM fourfolds 
are deformations of the derived category of a K3 surface. 
Yet, as mentioned above, for very general $X$ these categories are not 
equivalent to the derived category of a K3 surface. 
There even exist $X$ such that $\cA_X$ is not equivalent to the \emph{twisted} 
derived category of a K3 surface
(see Remark~\ref{remark-not-K3-type}). 
Families of categories with these properties appear to be quite rare; 
this is the first example since~\cite{kuznetsov2010derived}.

Our second main result shows that the K3 categories attached to GM and cubic fourfolds are 
not only analogous, but in some cases even coincide.

\begin{theorem}
\label{theorem-intro-cubic}
Let $X$ be a generic ordinary GM fourfold containing a plane of type $\G(2,3)$. 
Then there is a cubic fourfold $X'$ such that $\cA_X \simeq \cA_{X'}$. 
\end{theorem}

For a more precise statement, see Theorem~\ref{theorem-associated-cubic}. 
The cubic fourfold $X'$ 
is given explicitly by a construction 
of Debarre--Iliev--Manivel~\cite{DIM4fold}.
In fact, $X'$ is birational to $X$ and we use the structure of this birational isomorphism
to establish the result. 
We note that GM fourfolds as in the theorem form a $21$-dimensional 
(codimension $3$ in moduli) family.
Theorem~\ref{theorem-intro-cubic} can be considered as a  
step toward a $4$-dimensional analogue of~\cite{kuznetsov-fano-3-folds}, which 
exhibits mysterious coincidences among the derived categories of Fano threefolds.

\subsection{Further directions}
\label{subsection-further-directions}
The above results relate the K3 categories attached to three 
different types of varieties: GM fourfolds, cubic fourfolds, and K3 surfaces 
(in the last case the K3 category is the whole derived category). 
We call two such varieties $X_1$ and $X_2$ \emph{derived partners} if their K3 categories are equivalent. 
There is also a notion of $X_1$ and $X_2$ being \emph{Hodge-theoretic partners}. 
Roughly, this means that there is an ``extra'' integral middle-degree 
Hodge class $\alpha_i$ on $X_i$, such that if $K_i \subset \rH^{\dim(X_i)}(X_i, \bZ)$ denotes the 
lattice generated by $\alpha_i$ and certain tautological algebraic cycles on $X_i$, 
then the orthogonals $K_1^{\perp}$ and $K_2^{\perp}$ are isomorphic as polarized 
Hodge structures (up to a Tate twist). 
This notion was studied in \cite{hassett}, \cite{DIM4fold}, under the 
terminology that ``$X_2$ is associated to $X_1$''. 
Using lattice theoretic techniques, countably many families of GM fourfolds with Hodge-theoretic K3 and cubic fourfold partners 
are produced in~\cite{DIM4fold}. 

We expect that a GM fourfold has a derived partner of a given type 
if and only if it has a Hodge-theoretic partner of the same type. 
Theorems~\ref{theorem-intro-K3} and~\ref{theorem-intro-cubic} can be thought of 
as evidence for this expectation, since by \cite[\S7.5 and \S7.2]{DIM4fold} a GM 
fourfold as in Theorem~\ref{theorem-intro-K3} or Theorem~\ref{theorem-intro-cubic} 
has a Hodge-theoretic K3 or cubic fourfold partner, respectively.    
Addington and Thomas~\cite{addington-thomas} proved (generically) the 
analogous expectation for K3 partners of cubic fourfolds. 
Their method is deformation theoretic, and requires as a starting point 
an analogue of Theorem~\ref{theorem-intro-K3} for cubic fourfolds. 

Finally, we note that there are some other Fano varieties which 
fit into the above story, i.e. whose derived category contains a K3 category. 
One example is provided by a hyperplane section of the Grassmannian $\G(3,10)$, see~\cite{debarre-voisin} for a discussion of related geometric questions and~\cite[Corollary~4.4]{kuznetsov2015calabi} for the construction of a K3 category.
To find other examples, one can use available classification results for Fano fourfolds.
In~\cite{kuchle} K\"uchle classified 
Fano fourfolds of index $1$ which can be represented as zero 
loci of equivariant vector bundles on Grassmannians. 
Among these, three types --- labeled (c5), (c7), and (d3) in~\cite{kuchle} --- 
have middle Hodge structure of K3 type. In~\cite{kuznetsov2015kuchle} 
it was shown that fourfolds of type (d3) are isomorphic to the blowup of~$\bP^1\times\bP^1\times\bP^1\times\bP^1$ in a K3 surface, 
and those of type (c7) are isomorphic to the blowup of a cubic fourfold 
in a Veronese surface. 
In particular, these fourfolds do indeed have a K3 category in their 
derived category, but they reduce to known examples. 
Fourfolds of type (c5), however, conjecturally give rise to genuinely new 
K3 categories (see~\cite{kuznetsov2016c5} for a discussion of the geometry of these 
fourfolds). 

\subsection{Organization of the paper} 
In \S\ref{section-GM-varieties}, we define GM categories and study their basic properties. 
After recalling some facts about GM varieties in \S\ref{subsection-classification}, 
we define GM categories in~\S\ref{section-AX}. 
In~\S\ref{subsection-serre-functor}--\S\ref{subsection-geometricity}, we study 
some basic invariants of GM categories (Serre functors, Hochschild homology 
and cohomology, and Grothendieck groups) and as an application show that 
GM categories are usually not equivalent to the derived category of a variety. 
In~\S\ref{subsection:self-duality} we show that GM categories are self-dual, i.e. 
admit an equivalence with the opposite category.

In \S\ref{section-conjectures}, we formulate our conjectures about the duality and 
rationality of GM varieties. 
The preliminary~\S\ref{subsection-EPW} recalls from~\cite[\S3]{debarre-kuznetsov} a description 
of the set of isomorphism classes of GM varieties in terms of Lagrangian data. 
In~\S\ref{subsection:duality} we state the duality conjecture and discuss its consequences, 
and in~\S\ref{subsection-rationality-conjecture} we discuss the rationality conjecture.
This section is independent of the material in \S\ref{subsection-serre-functor}--\S\ref{subsection:self-duality}. 

The purpose of \S\ref{section-associated-K3} is to prove Theorem~\ref{theorem-associated-K3}, 
a more precise version of Theorem~\ref{theorem-intro-K3} from above. 
The statement of Theorem~\ref{theorem-associated-K3} requires the 
duality terminology introduced in \S\ref{subsection:duality}. 
However, in \S\ref{subsection-setup-associated-K3} we translate Theorem~\ref{theorem-associated-K3} 
into a statement that does not involve this terminology. 
From then on, \S\ref{section-associated-K3} can be read independently from the rest of the paper.  

The goal of \S\ref{section-associated-cubic} is to prove Theorem~\ref{theorem-associated-cubic}, 
a more precise version of Theorem~\ref{theorem-intro-cubic} from above. 
This section can also be read independently from the rest of the paper. 

Finally, in Appendix~\ref{appendix-moduli}, we prove that GM varieties of a fixed dimension form 
a smooth and irreducible Deligne--Mumford stack, whose dimension we compute. 
This result is not used in an essential way in the body of the paper, but it is 
psychologically useful.

\subsection{Notation and conventions} 
\label{subsection:notation}
We work over an algebraically closed field $\bC$ of characteristic~$0$.
A variety is an integral, separated scheme of finite type over~$\bC$. 
A vector bundle on a variety $X$ is a finite locally free $\cO_{X}$-module. 
The projective bundle of a vector bundle $\cE$ on a variety $X$ is
\begin{equation*}
\xymatrix{
\bP(\cE) = \Proj(\Sym^{\bullet}(\cE^{\vee})) \ar[r]^-{\pi} & X,
}
\end{equation*}
with $\cO_{\bP(\cE)}(1)$ normalized so that $\pi_*\cO_{\bP(\cE)}(1) = \cE^{\vee}$. 
We often commit the following convenient abuse of notation: 
given a divisor class $D$ on a variety $X$, we denote still by $D$ its pullback to any variety mapping to $X$. 
Throughout the paper, we use $V_n$ to denote an $n$-dimensional vector space.
We denote by \mbox{$\bG = \G(2,V_5)$} the Grassmannian of $2$-dimensional subspaces of~$V_5$. 

In this paper, triangulated categories are $\bC$-linear and functors between 
them are $\bC$-linear and exact. 
For a variety $X$, by the \emph{derived category} $\Db(X)$ we mean 
the bounded derived category of coherent sheaves on $X$, regarded as a triangulated category. 
For a morphism of varieties $f\colon X \to Y$, we write  
$f_*\colon \Db(X) \to \Db(Y)$
for the derived pushforward (provided~$f$ is proper), and 
$f^*\colon \Db(Y) \to \Db(X)$
for the derived pullback (provided $f$ has finite $\Tor$-dimension). 
For $\cF, \cG \in \Db(X)$, we write $\cF \otimes \cG$ for the derived tensor product.

We write 
$\cT = \langle \cA_1, \dots, \cA_n \rangle$ 
for a semiorthogonal decomposition of a triangulated category~$\cT$ with components $\cA_1, \dots, \cA_n$.
For an admissible subcategory $\cA \subset \cT$ we write
\begin{align*}
\cA^{\perp} & = \left \{ \cF \in \cT ~ | ~ \Hom(\cG, \cF) = 0 \text{ for all } \cG \in \cA \right \},  \\
{}^{\perp}\cA & =  \left \{ \cF \in \cT ~ | ~ \Hom(\cF, \cG) = 0 \text{ for all } \cG \in \cA \right \}, 
\end{align*} 
for its right and the left orthogonals, so that we have $\cT = \langle \cA^\perp, \cA \rangle = \langle \cA, {}^\perp\cA \rangle$.

We regard graded vector spaces as complexes with trivial differential, 
so that any such vector space can be written as 
$W_\bullet = \bigoplus_{n} W_n[-n]$, where $W_n$ denotes the degree $n$ piece. 
We often suppress the degree $0$ shift $[0]$ from our notation.

\subsection{Acknowledgements}
We would like to thank Olivier Debarre, Joe Harris, and Johan de Jong for 
many useful discussions. 
We are also grateful to Daniel Huybrechts, Richard Thomas, and Ravi Vakil for comments and questions. 
Finally, we thank the referee for suggestions about the presentation of the paper.


\section{GM categories} 
\label{section-GM-varieties}

In this section, we define GM categories and study their basic properties. 
We start with a quick review of the key features of GM varieties.

\subsection{GM varieties}
\label{subsection-classification}

Let $V_5$ be a $5$-dimensional vector space and $\bG = \G(2,V_5)$ the Grassmannian of $2$-dimensional subspaces. 
Consider the Pl\"{u}cker embedding $\bG \hookrightarrow \bP(\wedge^2V_5)$ and
let $\Cone(\bG) \subset \bP(\bC \oplus \wedge^2V_5)$ be the cone over $\bG$.
Further, let 
\begin{equation*}
W \subset \bC \oplus \wedge^2V_5
\end{equation*}
be a linear subspace of dimension $n+5$ with $2 \le n \le 6$, and $Q \subset \bP(W)$ a quadric hypersurface.
By Definition~\ref{definition-GM}, if the intersection 
\begin{equation}
\label{x:gm:general}
X = \Cone(\bG) \cap Q
\end{equation}
is smooth and transverse, then $X$ is a GM variety of dimension $n$, 
and every GM variety can be written in this form.

There is a natural polarization $H$ on a GM variety $X$, 
given by the restriction of the hyperplane class on $\bP(\bC \oplus \wedge^2V_5)$; 
we denote by $\cO_X(1)$ the corresponding line bundle on $X$. 
It is straightforward to check that  
\begin{equation} \label{eq:omega-x}
H^n = 10 \quad \text{and} \quad -K_X = (n-2)H. 
\end{equation}
Moreover, we have
\begin{equation}
\label{eq:pic-x}
\begin{aligned}
&\text{if $\dim(X) \ge 3$}, 	&& \text{then $\Pic(X) = \bZ H$, and}\\
&\text{if $\dim(X) = 2$},   		&& \text{then $(X,H)$ is a Brill--Noether general K3 surface.}
\end{aligned}
\end{equation} 
Conversely, by~\cite[Theorem~2.16]{debarre-kuznetsov} any smooth projective polarized variety of dimension $\ge 2$ satisfying~\eqref{eq:pic-x} and~\eqref{eq:omega-x} is a GM variety.

The intersection $\Cone(\bG) \cap Q$ does not contain 
the vertex of the cone, since $X$ is smooth. 
Hence projection from the vertex defines a regular map 
\begin{equation*}
f \colon X \to \bG,
\end{equation*}
called the \emph{Gushel map}. 
Let $\cU$ be the rank $2$ tautological subbundle on $\bG$. 
Then $\cU_X = f^* \cU$ is a rank 2 vector bundle on $X$, called the \emph{Gushel bundle}. 
By~\cite[\S2.1]{debarre-kuznetsov}, the Gushel map and the Gushel bundle are 
canonically associated to $X$, i.e. only depend on the abstract 
polarized variety $(X,H)$ and not on the particular realization \eqref{x:gm:general}.
In particular, so is the space $V_5$ (being the dual of the space of sections of~$\cU_X^{\vee}$), 
and we will sometimes write it as $V_5(X)$ to emphasize this. 
The space $W$ is also canonically associated to $X$, since its dual is the 
space of global sections of $\cO_X(1)$. 
The quadric $Q$, however, is not canonically associated to $X$, see~\S\ref{subsection-EPW}.  

The intersection 
\begin{equation*}
M_X = \Cone(\bG) \cap \bP(W)
\end{equation*}
is called the \emph{Grassmannian hull} of $X$. 
Note that $X = M_X \cap Q$ is a quadric section of $M_X$. 
Let $W'$ be the projection of $W$ to $\wedge^2V_5$. 
The intersection
\begin{equation}
\label{pghull} 
M'_X = \bG \cap \bP(W')
\end{equation}
is called the \emph{projected Grassmannian hull} of $X$. 
Again by~\cite{debarre-kuznetsov}, both $M_X$ and $M'_X$ are canonically 
associated to $X$. 

The Gushel map is either an embedding or a double covering of $M'_X$, 
according to whether the projection map $W \to W'$ is an isomorphism or has $1$-dimensional kernel.
In the first case, $W \cong W'$ and $M_X \cong M'_X$. 
Then considering $Q$ as a subvariety of 
$\bP(W')$, we have
\begin{equation}\label{eq:gm:ordinary}
X \cong M'_X \cap Q. 
\end{equation}
That is, $X$ is a quadric section of a linear section of the Grassmannian $\bG$. 
A GM variety of this type is called \emph{ordinary}.

If the map $W \to W'$ has $1$-dimensional kernel, then we have $\bP(W) = \Cone(\bP(W'))$ and~$M_X = \Cone(M'_X)$. 
As $Q$ does not contain the vertex of the cone (by smoothness of~$X$), projection from the vertex gives a double cover 
\begin{equation}\label{eq:gm:special}
X  \xrightarrow{\ 2:1\ } M'_X. 
\end{equation}
That is, $X$ is a double cover of a linear section of the Grassmannian $\bG$. 
A GM variety of this type is called \emph{special}. 

\begin{lemma}[\protect{\cite[Proposition~2.22]{debarre-kuznetsov}}]
\label{lemma:gr-hull}
Let $X$ be a GM variety of dimension $n$. Then the intersection~\eqref{pghull} 
defining $M'_X$ is dimensionally transverse. Moreover: 
\begin{enumerate}
\item If $n \geq 3$, or if $n=2$ and $X$ is special, then $M'_X$ is smooth.  
\item If $n = 2$ and $X$ is ordinary, then $M'_X$ has at worst rational double point singularities.  
\end{enumerate}
\end{lemma}

By Lemma~\ref{lemma:gr-hull}, 
if $X$ is special then $M'_X$ is smooth. 
Further, by~\cite[\S2.5]{debarre-kuznetsov} the branch divisor of the double cover~\eqref{eq:gm:special} is
the smooth intersection $X' = \bG \cap Q'$, 
where $Q' = Q \cap \bP(W')$ is a quadric hypersurface in $\bP(W')$. 
Hence, as long as $n \geq 3$, the branch divisor 
$X'$ of~\eqref{eq:gm:special} is an ordinary GM variety of dimension $n-1$. 
This gives rise to an operation taking a GM variety of one type to the opposite type, 
by defining in this situation
\begin{equation} 
\label{opp-variety}
X^{\opp} = X' \qquad \text{and} \qquad
(X')^{\opp} = X. 
\end{equation}
Note that we have $\dim X^{\opp} = \dim X \pm 1$.
The opposite GM variety is not defined for special GM surfaces. 

\subsection{Definition of GM categories} 
\label{section-AX}

By the discussion in \S\ref{subsection-classification},
any GM variety $X$ of dimension $n \ge 3$ is obtained from a smooth linear section $M'_X$ 
of $\bG$ by taking a quadric section or a branched double cover. 
To describe a natural semiorthogonal decomposition of $\Db(X)$, we first recall that 
$\bG$ and its smooth linear sections of dimension at least~$3$ admit rectangular Lefschetz decompositions 
(in the sense of~\cite[\S4]{kuznetsov2007HPD})
of their derived categories. 
Note that the bundles $\cO_{\bG}, \cU^\vee$ form an exceptional pair in $\Db(\bG)$, 
where recall $\cU$ denotes the tautological rank $2$ bundle. Let 
\begin{equation}
\label{equation-definition-B}
\cB = \langle \cO_{\bG}, \cU^\vee \rangle \subset \Db(\bG)
\end{equation}
be the triangulated subcategory they generate. 
The following result holds by~\cite[\S6.1]{kuznetsov2006hyperplane}.

\begin{lemma}
\label{lemma-sod-linear-section}
Let $M$ be a smooth linear section of $\bG \subset \bP(\wedge^2V_5)$ of dimension $N \geq 3$. 
Let $i\colon M \hookrightarrow \bG$ be the inclusion. 
\begin{enumerate}
\item The functor $i^*\colon \Db(\bG) \rightarrow \Db(M)$ is fully faithful on ${\cB \subset \Db(\bG)}$.
\item Denoting the essential image of $\cB$ by $\cB_M$, there is a semiorthogonal
decomposition
\begin{equation}
\label{dbm}
\Db(M) = \langle \cB_M, \cB_M(1), \dots, \cB_M(N-2) \rangle.
\end{equation}
\end{enumerate}
\end{lemma}

The next result gives a semiorthogonal decomposition of the derived category of a GM variety. 

\begin{proposition}
\label{proposition-dbx}
Let $X$ be a GM variety of dimension $n \geq 3$. Let $f\colon X \to \bG$ be the Gushel map. 
\begin{enumerate}
\item The functor $f^*\colon \Db(\bG) \to \Db(X)$ is fully faithful on ${\cB \subset \Db(\bG)}$.
\item Denoting the essential image of $\cB$ by $\cB_X$, so that $\cB_X = \langle \cO_X, \cU^\vee_X \rangle$, there is a semiorthogonal decomposition 
\begin{equation}
\label{dbx}
\Db(X) = \langle \cA_X, \cB_X, \cB_X(1), \dots, \cB_X(n-3) \rangle,
\end{equation}
where $\cA_X$ is the right orthogonal category to 
$\langle \cB_X, \dots, \cB_X(n-3) \rangle \subset \Db(X)$.
\end{enumerate}
Thus $\Db(X)$ has a semiorthogonal decomposition with the category $\cA_X$ and $2(n-2)$ exceptional objects as components. 
\end{proposition}

\begin{remark}
If $n = 2$ we set $\cA_X = \Db(X)$, so that~\eqref{dbx} still holds. 
\end{remark}

\begin{proof}
The Gushel map factors through the map 
$X \to M'_X$ to the projected Grassmannian hull~$M'_X$ defined by~\eqref{pghull}. 
By Lemma~\ref{lemma:gr-hull}, $M'_X$ is smooth and has dimension $n+1$ if 
$X$ is ordinary, or dimension $n$ if $X$ is special. 
In particular, $\Db(M'_X)$ has a semiorthogonal decomposition of the form~\eqref{dbm}. 
Further, $X \to M'_X$ realizes $X$ as a quadric section~\eqref{eq:gm:ordinary} 
if $X$ is ordinary, or as a double cover~\eqref{eq:gm:special} if $X$ is special. 
Now applying~\cite[Lemmas 5.1 and 5.5]{cyclic-covers} gives the result. 
\end{proof}

\begin{definition}
Let $X$ be a GM variety. The \emph{GM category} of $X$ is the 
category $\cA_X$ defined by the semiorthogonal decomposition~\eqref{dbx}.  
\end{definition}

More explicitly, using the definition~\eqref{equation-definition-B} of $\cB$, 
the defining semiorthogonal decomposition of a GM category $\cA_X$ can 
be written as 
\begin{equation}
\label{eq:dbx-detailed}
\Db(X) = \langle \cA_X, \cO_X, \cU_X^{\vee}, \dots, \cO_X(n-3), \cU_X^{\vee}(n-3) \rangle.
\end{equation}

The GM category $\cA_X$ is the main object of study of this paper. 
As we will see below, its properties depend strongly on the parity of $\dim(X)$. 
For this reason, we sometimes emphasize the parity of $\dim(X)$ by calling 
$\cA_X$ an \emph{even} or \emph{odd GM category} according to 
whether $\dim(X)$ is even or odd.

\subsection{Serre functors of GM categories} 
\label{subsection-serre-functor}
Recall from \cite{bondal-kapranov} that a \emph{Serre functor} for a 
triangulated category $\cT$ is an autoequivalence
$\rS_\cT\colon\cT \to \cT$ with bifunctorial isomorphisms 
\begin{equation*}
 \Hom(\cF,\rS_\cT(\cG)) \cong \Hom(\cG,\cF)^\vee
\end{equation*}
for $\cF,\cG \in \cT$. If a Serre functor exists, it is unique.
If $X$ is a smooth proper variety, then $\Db(X)$ has a Serre functor given by the formula
\begin{equation}
\label{eq:serre-x}
\rS_{\Db(X)}(\cF) = \cF \otimes \omega_X[\dim X].
\end{equation}
Moreover, given an admissible subcategory $\cA \subset \cT$, if 
$\cT$ admits a Serre functor then so does~$\cA$. 
Using~\cite{kuznetsov2015calabi}, we can describe the Serre functor of a GM category.  

\begin{proposition}
\label{proposition-serre-functor}
Let $X$ be a GM variety of dimension $n$.
\begin{enumerate}
\item If $n$ is even, the Serre functor of the GM category $\cA_X$ satisfies $\rS_{\cA_X} \cong [2]$.
\item If $n$ is odd, the Serre functor of the GM category $\cA_X$ satisfies $\rS_{\cA_X} \cong \sigma \circ [2]$ 
for a nontrivial involutive autoequivalence $\sigma$ of $\cA_X$. If in addition $X$ is special, 
then $\sigma$ is induced by the involution of the double cover~\eqref{eq:gm:special}.
\end{enumerate}
\end{proposition}

\begin{proof}
If $n = 2$, then $\cA_X = \Db(X)$ and $X$ is a K3 surface, so the result holds by~\eqref{eq:serre-x}. 
If $n \geq 3$, then as in the proof of Proposition~\ref{proposition-dbx} we may express 
$X$ as a quadric section or double cover of the smooth variety $M'_X$. 
It is easy to see the length $m$ of the semiorthogonal
decomposition of $\Db(M'_X)$ 
given by Lemma~\ref{lemma-sod-linear-section} satisfies $K_{M'_X} = -mH$, where $H$ 
is the restriction of the ample generator of $\Pic(\bG)$. 
Hence we may apply~\cite[Corollaries 3.7 and 3.8]{kuznetsov2015calabi} to see that the Serre 
functors have the desired form. 

If $\sigma$ were trivial, then the Hochschild homology $\HH_{-2}(\cA_X)$ would 
be nontrivial (see Proposition~\ref{proposition-HH-CY}), 
which contradicts the computation of Proposition~\ref{proposition-HH_*} below.
\end{proof}

Proposition~\ref{proposition-serre-functor} shows that even GM categories can be regarded 
as ``noncommutative K3 surfaces'', and odd GM categories can be regarded 
as ``noncommutative Enriques surfaces''.  
This analogy goes further than the relation between Serre functors. 
For instance, any Enriques surface (in characteristic $0$) is the quotient 
of a K3 surface by an involution.
Similarly, the results of~\cite{cyclic-covers} show that 
odd GM categories can be described as ``quotients'' of even GM categories by involutions.
To state this precisely, recall from~\S\ref{subsection-classification} that unless $X$ is a special GM surface,
there is an associated GM variety 
$X^{\opp}$ of the opposite type and parity of dimension. 
The following result is proved in \cite[\S8.2]{cyclic-covers}.

\begin{proposition}
\label{proposition-GM-category-quotient}
Let $X$ be a GM variety which is not a special GM surface.
Then there is a~$\bZ/2$-action on the GM category~$\cA_{X}$ such 
that if $\cA_{X}^{\bZ/2}$ denotes the equivariant category, 
then there is an equivalence 
\begin{equation*}
\cA_{X}^{\bZ/2} \simeq \cA_{X^{\opp}} . 
\end{equation*}
If $\sigma$ is the autoequivalence generating the $\bZ/2$-action on $\cA_X$, then 
$\sigma$ is induced by the involution of the double covering $X \to M'_X$ if $X$ is special, 
and $\sigma = \rS_{\cA_X} \circ [-2]$ if $\dim(X)$ is odd.
\end{proposition} 

\subsection{Hochschild homology of GM categories}
\label{section-hochschild-homology}
Given a suitably enhanced triangulated category $\cA$, there is 
an invariant $\HH_\bullet(\cA)$ called its \emph{Hochschild homology}, 
which is a graded $\bC$-vector space. 
We will exclusively be interested in admissible subcategories of the 
derived category of a smooth projective variety. 
For a definition of Hochschild homology in this context, see~\cite{kuznetsov2009hochschild}. 

If $\cA = \Db(X)$, we write $\HH_{\bullet}(X)$ for $\HH_{\bullet}(\cA)$. 
The Hochschild--Kostant--Rosenberg (HKR) isomorphism gives the 
following explicit description of Hochschild homology in this case \cite{markarian}:
\begin{equation}
\label{eq:HKR}
\HH_{i}(X) \cong \bigoplus_{q - p = i} \rH^{q}(X, \Omega^{p}_{X}).
\end{equation} 

An important property of Hochschild homology is that it is additive under 
semiorthogonal decompositions. 
\begin{theorem}[\protect{\cite[Theorem~7.3]{kuznetsov2009hochschild}}]
\label{theorem:additivity}
Let $X$ be a smooth projective variety. 
Given a semiorthogonal decomposition $\Db(X) = \langle \cA_1,\cA_2,\dots,\cA_m \rangle$, 
there is an isomorphism 
\begin{equation*}
 \HH_\bullet(X) \cong \bigoplus_{i=1}^m \HH_\bullet(\cA_i).
\end{equation*}
\end{theorem}

By combining this additivity property with the HKR isomorphism for GM varieties, we can 
compute the Hochschild homology of GM categories. 

\begin{proposition}
\label{proposition-HH_*}
Let $X$ be a GM variety of dimension $n$. Then 
\begin{equation*}
\HH_\bullet(\cA_X) \cong  
\left \{
\begin{array}{c@{\hspace{4pt}}c@{\hspace{4pt}}c@{\hspace{4pt}}c@{\hspace{4pt}}cl}
\bC[2] & \oplus & \bC^{22} & \oplus & \bC[-2] & \text{if $n$ is even,} \\
\bC^{10}[1] & \oplus & \bC^2 & \oplus & \bC^{10}[-1] & \text{if $n$ is odd.}
\end{array}
\right.
\end{equation*}
\end{proposition}

\begin{proof}
By~\eqref{eq:dbx-detailed} there is a semiorthogonal decomposition of 
$\Db(X)$ with $\cA_X$ and~$2(n-2)$ exceptional objects as components. 
Since the category generated by an exceptional object is equivalent 
to the derived category of a point, its Hochschild homology is just $\bC$. 
Hence by additivity, 
\begin{equation*}
\HH_{\bullet}(X) \cong \HH_{\bullet}(\cA_X) \oplus \bC^{2(n-2)}. 
\end{equation*}
By~{\eqref{eq:HKR}} the graded dimension of $\HH_{\bullet}(X)$ 
can be computed by summing the columns of the Hodge diamond of $X$, which looks as follows 
(see~\cite{logachev2012fano}, \cite{iliev2011fano}, \cite{nagel1998generalized}, and~\cite{debarre-kuznetsov-periods}): 
\begin{equation*}
{\tabulinesep=1.2mm
\begin{tabu}{|c|c|c|c|c|c}
\hline
\dim(X) = 2 & \dim(X) = 3 & \dim(X) = 4 & \dim(X) = 5 & \dim(X) = 6 \\
\hline 
\begin{smallmatrix}
&& 1 \\
& 0 && 0 \\
1 && 20 && 1 \\ 
& 0 && 0 \\
&& 1 
\end{smallmatrix}
& 
\begin{smallmatrix}
&&& 1 \\
&& 0 && 0 \\
& 0 && 1 && 0 \\
0 && 10 && 10 && 0 \\
& 0 && 1 && 0 \\
&& 0 && 0 \\
&&& 1 
\end{smallmatrix}
&
\begin{smallmatrix}
&&&& 1 \\
&&& 0 && 0 \\
&& 0 && 1 && 0 \\
& 0 && 0 && 0 && 0 \\
0 && 1 && 22 && 1 && 0 \\ 
& 0 && 0 && 0 && 0 \\
&& 0 && 1 && 0 \\
&&& 0 && 0 \\
&&&& 1 
\end{smallmatrix}
&
\begin{smallmatrix}
&&&&& 1 \\
&&&& 0 && 0 \\
&&& 0 && 1 && 0 \\
&& 0 && 0 && 0 && 0 \\
& 0 && 0 && 2 && 0 && 0 \\
0 && 0 && 10 && 10 && 0 && 0 \\
& 0 && 0 && 2 && 0 && 0 \\
&& 0 && 0 && 0 && 0 \\
&&& 0 && 1 && 0 \\
&&&& 0 && 0 \\
&&&&& 1 
\end{smallmatrix}
&
\begin{smallmatrix}
&&&&&& 1 \\
&&&&& 0 && 0 \\
&&&& 0 && 1 && 0 \\
&&& 0 && 0 && 0 && 0 \\
&& 0 && 0 && 2 && 0 && 0 \\
& 0 && 0 && 0 && 0 && 0 && 0 \\
0 && 0 && 1 && 22 && 1 && 0 && 0 \\ 
& 0 && 0 && 0 && 0 && 0 && 0 \\
&& 0 && 0 && 2 && 0 && 0 \\
&&& 0 && 0 && 0 && 0 \\
&&&& 0 && 1 && 0 \\
&&&&& 0 && 0 \\
&&&&&& 1 
\end{smallmatrix} \\
\hline
\end{tabu}}
\end{equation*}
Now the lemma follows by inspection. 
\end{proof}

\subsection{Hochschild cohomology of GM categories}
\label{section-hochschild-cohomology}
Given a suitably enhanced triangulated category $\cA$,  
there is also an invariant $\HH^\bullet(\cA)$ called its \emph{Hochschild cohomology}, 
which has the structure of a graded $\bC$-algebra.  
Again, for a definition in the case where $\cA$ is an admissible subcategory 
of the derived category of a smooth projective variety, see~\cite{kuznetsov2009hochschild}. 

If $\cA = \Db(X)$, we write $\HH^{\bullet}(X)$ for $\HH^{\bullet}(\cA)$. 
There is the following version of the HKR isomorphism 
for Hochschild cohomology~\cite{markarian}:
\begin{equation}
\label{eq:HKR-cohomology}
\HH^i(X) \cong \bigoplus_{p + q = i} \rH^q(X, \wedge^p\rT_X).
\end{equation} 

Hochschild cohomology is not additive under semiorthogonal decompositions, 
and so it is generally much harder to compute than Hochschild homology. 
There is, however, a case when the computation simplifies considerably.
Recall that a triangulated category $\cA$ is called 
\emph{$n$-Calabi--Yau} if the shift functor $[n]$ is a Serre functor for $\cA$. 

\begin{proposition}[\protect{\cite[Proposition~5.2]{kuznetsov2015calabi}}]
\label{proposition-HH-CY}
Let $\cA$ be an admissible subcategory of $\Db(X)$ for a smooth projective variety $X$. 
If $\cA$ is an $n$-Calabi--Yau category, 
then for each $i$ there is an isomorphism of vector spaces
\begin{equation*}
\HH^i(\cA) \cong \HH_{i-n}(\cA).
\end{equation*}
\end{proposition}

This immediately applies to even GM categories, as by Proposition~\ref{proposition-serre-functor} 
they are $2$-Calabi--Yau. 

\begin{corollary}
\label{corollary-even-GM-HC}
Let $X$ be a GM variety of even dimension. Then 
\begin{equation*}
\HH^{\bullet}(\cA_X) \cong \bC \oplus \bC^{22}[-2] \oplus \bC[-4]. 
\end{equation*}
\end{corollary}

The Hochschild cohomology of odd GM categories is significantly harder 
to compute. 
Our strategy is to exploit the fact that there is a $\bZ/2$-action on such 
a category, with invariants an even GM category. 
By the results of~\cite{equivariant-HC}, this reduces us to a problem 
involving the Hochschild cohomology of an even GM category and 
the Hochschild \emph{homology} of an odd GM category.

\begin{proposition}
\label{proposition-odd-GM-HC}
Let $X$ be a GM variety of odd dimension. Then 
\begin{equation*}
\HH^{\bullet}(\cA_X) \cong \bC \oplus \bC^{20}[-2] \oplus \bC[-4]. 
\end{equation*}
\end{proposition}

\begin{proof}
Recall that by Proposition~\ref{proposition-GM-category-quotient} there is a $\bZ/2$-action on $\cA_X$ such that if 
$\sigma\colon \cA_X \to \cA_X$ denotes the corresponding 
involutive autoequivalence, then: 
\begin{enumerate}
\item
\label{serre-recap} $\rS_{\cA_X} = \sigma \circ [2]$ is a Serre functor for $\cA_X$.
\item
\label{involution-recap} 
$\cA_X^{\bZ/2} \simeq \cA_{X^\op}$, where $X^{\op}$ is the opposite variety to $X$.
\end{enumerate}
As stated, these are results at the level of triangulated categories, 
but they also hold at the enhanced level. 
Namely, in the terminology of~\cite{equivariant-HC}, there is a $\bC$-linear 
stable $\infty$-category $\Db(X)^{\enh}$ (denoted $\Perf(X)$ in~\cite{equivariant-HC}) 
with homotopy category $\Db(X)$. 
The category $\Db(X)^{\enh}$ admits a semiorthogonal decomposition of the same 
form as~\eqref{dbx}, which defines a $\bC$-linear stable $\infty$-category $\cA_X^{\enh}$ 
whose homotopy category is $\cA_X$. 
If $\sigma^{\enh}\colon \cA_X^{\enh} \to \cA_X^{\enh}$ denotes the corresponding involutive autoequivalence, then \eqref{serre-recap} 
and \eqref{involution-recap} above hold with $\cA_X$, $\rS_{\cA_X}$, $\sigma$, and~$\cA_{X^\op}$ replaced by their enhanced versions, 
and \eqref{serre-recap} and \eqref{involution-recap} are recovered by 
passing to homotopy categories. 
The Hochschild (co)homology of $\cA_X$ and $\cA_{X^\opp}$ agree with 
the corresponding invariants of their enhancements. 
Hence~\cite[Corollary~1.3]{equivariant-HC} gives 
\begin{equation}
\label{equation-odd-GM-HC}
\HH^\bullet(\cA_{X^\opp}) \cong 
\HH^\bullet(\cA_X) \oplus (\HH_\bullet(\cA_X)^{\bZ/2}[-2])
\end{equation}
where the $\bZ/2$-action on $\HH_\bullet(\cA_X)$ is induced by $\sigma$. 

Since $X$ has odd dimension (and hence $X^{\opp}$ has even dimension), 
by Corollary~\ref{corollary-even-GM-HC} we have 
\begin{equation*}
\HH^{\bullet}(\cA_{X^\opp}) \cong \bC \oplus \bC^{22}[-2] \oplus \bC[-4],  
\end{equation*}
and by Proposition~\ref{proposition-HH_*} we have 
\begin{equation*}
\HH_\bullet(\cA_X) \cong \bC^{10}[1] \oplus \bC^2 \oplus \bC^{10}[-1]. 
\end{equation*}
Combined with~\eqref{equation-odd-GM-HC}, this immediately shows 
$\HH_{\bullet}(\cA_X)^{\bZ/2}$ is concentrated in degree $0$, i.e. 
$\HH_{\bullet}(\cA_X)^{\bZ/2} \cong \bC^d$ for some $0 \leq d \leq 2$, 
and  
\begin{equation*}
\HH^{\bullet}(\cA_X) \cong \bC \oplus \bC^{22-d}[-2] \oplus \bC[-4]. 
\end{equation*} 

To prove $d = 2$, we apply~\cite[Corollary 3.11]{polishchuk2014}, 
which gives an equality 
\begin{equation}
\label{equation-LFP}
\sum_i (-1)^i \dim \HH^i(\cA_X) = 
\sum_i (-1)^i \, \Tr((\rS_{\cA_X}^{-1})_*\colon \HH_i(\cA_X) \to \HH_i(\cA_X)). 
\end{equation}
Note that since $\rS_{\cA_X} = \sigma \circ [2]$,  
the map $(\rS^{-1}_{\cA_X})_* \colon \HH_i(\cA_X) \to \HH_i(\cA_X)$ induced 
by $\rS^{-1}_{\cA_X}$ on Hochschild homology coincides with the map 
induced by $\sigma$, and in particular squares to the identity. 
It follows that the right side of~\eqref{equation-LFP} is bounded above 
by \mbox{$\sum_i \dim \HH_i(\cA_X) = 22$}. 
But the left side of~\eqref{equation-LFP} equals $24-d$ where 
$0 \leq d \leq 2$, so $d = 2$. 
\end{proof}

\begin{remark}
As a byproduct, the above proof shows that $\rS_{\cA_X}$ acts on $\HH_i(\cA_X)$ by $(-1)^i$ for any GM category $\cA_X$. 
\end{remark}

\begin{remark}
It is possible to show $d = 2$ in the above proof without appealing to the equality~\eqref{equation-LFP}, as follows. 
Note that the statement is deformation invariant, since it is equivalent to the Euler characteristic $\sum_i (-1)^i \dim \HH^i(\cA_X)$ being $22$. 
So we may assume $X$ is special. 
Then the $\bZ/2$-action on $\cA_X$ is induced by the involution of the double cover $X \to M'_X$. 
We want to show that $\bZ/2$ acts trivially on $\HH_0(\cA_X)$.  
But $\HH_\bullet(\cA_X)$ is canonically a summand of $\HH_{\bullet}(X)$, and we claim that the involution of the double cover acts trivially on~$\HH_0(X)$. 
Indeed, since $X$ is odd-dimensional, pullback under $X \to M'_X$ induces a surjection on even-degree cohomology and hence on $\HH_0$. 
The claim follows. 
\end{remark}

\begin{remark}
Proposition~\ref{proposition-odd-GM-HC} can also be deduced from Conjecture~\ref{conjecture:duality} stated below.
Indeed, the conjecture implies that the GM category of any GM variety of odd dimension 
is equivalent to that of an ordinary GM threefold, whose Hochschild cohomology 
can be computed using~\cite[Theorem~8.8]{kuznetsov2009hochschild}.
Yet another method for computing the Hochschild cohomology of GM categories is via the normal Hochschild cohomology spectral sequence of~\cite{kuznetsov2015height}, 
but this method becomes long and complicated for GM varieties of dimension bigger than $3$. 
\end{remark}

As an application, we discuss the indecomposability of GM categories. 
Recall that a triangulated category $\cT$ is called \emph{indecomposable} if it admits no nontrivial 
semiorthogonal decompositions, i.e. if $\cT = \langle \cA_1, \cA_2 \rangle$ implies either 
$\cA_1 \simeq 0$ or $\cA_2 \simeq 0$.
In general, there are very few techniques for proving indecomposability of a triangulated category.
However, for Calabi--Yau categories, we recall a simple criterion below. 

If $\cA$ is an admissible subcategory of the derived category of a smooth projective variety, 
we say $\cA$ is \emph{connected} if $\HH^0(\cA) = \bC$ (see~\cite[\S5.2]{kuznetsov2015calabi}).  
By Corollary~\ref{corollary-even-GM-HC} and Proposition~\ref{proposition-odd-GM-HC}, 
all GM categories are connected. 

\begin{proposition}[\protect{\cite[Proposition 5.5]{kuznetsov2015calabi}}] 
\label{proposition-indecomposable}
Let $\cA$ be a connected admissible subcategory of the derived category of a smooth 
projective variety.
Then $\cA$ admits no nontrivial completely orthogonal decompositions. 
If furthermore $\cA$ is Calabi--Yau, then $\cA$ is indecomposable.
\end{proposition}

\begin{corollary}
\label{corollary-AX-indecomposable}
Let $X$ be a GM variety of dimension $n$. 
\begin{enumerate}
\item If $n$ is even, then $\cA_X$ is indecomposable. 
\item If $n$ is odd, then $\cA_X$ admits no nontrivial completely orthogonal decompositions.
\end{enumerate}
\end{corollary}

\begin{proof}
This follows from Proposition \ref{proposition-indecomposable}, the connectivity of $\cA_{X}$, 
and the fact that $\cA_{X}$ is Calabi--Yau if $n$ is even.
\end{proof}

\begin{remark}
It is plausible that $\cA_X$ is indecomposable if $X$ is an odd-dimensional GM variety, 
but we do not know how to prove this. 
\end{remark}

\subsection{Grothendieck groups of GM categories}
\label{section-grothendieck-group}
The \emph{Grothendieck group} 
$\rK_{0}(\cT)$ of a triangulated category $\cT$ 
is the free group on isomorphism classes $[\cF]$ of 
objects $\cF \in \cT$, modulo the relations $[\cG] = [\cF] + [\cH]$ for every distinguished triangle
$\cF \rightarrow \cG \rightarrow \cH$. 

Assume $\cT$ is \emph{proper}, i.e. that $\bigoplus_i \Hom(\cF, \cG[i])$ is finite 
dimensional for all $\cF, \cG \in \cT$. 
For instance, this holds if $\cT$ is admissible in the derived category of a smooth projective variety. 
Then for $\cF, \cG \in \cT$, we set  
\begin{equation*}
\chi(\cF,\cG) = \sum_{i} (-1)^i \dim \Hom(\cF,\cG[i]).
\end{equation*}
This descends to a bilinear form $\chi\colon \rK_{0}(\cT) \times \rK_{0}(\cT) \rightarrow \bZ$, 
called the \emph{Euler form}. 
In general this form is neither symmetric nor antisymmetric. 
However, if $\cT$ admits a Serre functor 
(e.g. if $\cT$ is admissible in the derived category of a smooth projective variety), 
then the left and right kernels of the form $\chi$ agree, and we denote this common subgroup of $\rK_{0}(\cT)$ 
by $\ker(\chi)$. In this situation, the \emph{numerical Grothendieck group} is the quotient
\begin{equation*}
\rK_{0}(\cT)_{\num} = \rK_{0}(\cT)/\ker(\chi).
\end{equation*}
Note that $\rK_{0}(\cT)_{\num}$ is torsion free, since 
$\ker(\chi)$ is evidently saturated. 

If $X$ is a smooth projective variety, we write 
\begin{equation*}
\rK_{0}(X) = \rK_{0}(\Db(X)) 
\quad \text{and} \quad
\rK_{0}(X)_{\num} = \rK_{0}(\Db(X))_{\num}. 
\end{equation*}
Further, let $\CH(X)$ and $\CH(X)_{\num}$ denote the Chow rings of cycles 
modulo rational and numerical equivalence. 
The following well-known consequence of Hirzebruch--Riemann--Roch 
relates the (numerical) Grothendieck group of $X$ to its (numerical) Chow group. 

\begin{lemma}
\label{lemma-GRR}
Let $X$ be a smooth projective variety. 
Then there are isomorphisms 
\begin{equation*}
\rK_{0}(X) \otimes \bQ  \cong  \CH(X) \otimes \bQ \qquad \text{and} \qquad
\rK_{0}(X)_{\num} \otimes \bQ  \cong \CH(X)_{\num}\otimes \bQ. 
\end{equation*}
\end{lemma}

\begin{proof}
The isomorphisms are induced by the Chern character $\ch \colon \rK_{0}(X) \rightarrow \CH(X) \otimes \bQ$.  
For the first, see~\cite[Example 15.2.16(b)]{fulton}. 
The second then follows from the observation that, by Riemann--Roch, 
the kernel of the Euler form is precisely the preimage under the Chern character
of the subring of numerically trivial cycles. 
\end{proof}

The following well-known lemma says that Grothendieck groups are additive.  

\begin{lemma}
Let $X$ be a smooth projective variety. 
Given a semiorthogonal decomposition $\Db(X) = \langle \cA_1,\cA_2,\dots,\cA_m \rangle$, 
there are isomorphisms 
\begin{equation*}
\rK_0(X)  \cong  \bigoplus_{i = 1}^{m} \rK_0(\cA_i) \qquad \text{and} \qquad 
\rK_{0}(X)_{\num} \cong  \bigoplus_{i = 1}^{m} \rK_0(\cA_i)_{\num} . 
\end{equation*}
\end{lemma}

\begin{proof}
The embedding functors $\cA_i \hookrightarrow \Db(X)$ induce a map 
$\bigoplus_{i} \rK_0(\cA_i) \to \rK_0(X)$, whose inverse is the map induced 
by the projection functors $\Db(X) \to \cA_i$. This isomorphism also 
descends to numerical Grothendieck groups. 
\end{proof}

Now let $X$ be a GM variety. If $X$ is a surface then $\cA_X = \Db(X)$, so
the Grothendieck group of $\cA_X$ coincides with that of $X$. 
Below we describe $\rK_0(\cA_X)_{\num}$ if $X$ is odd dimensional, or if~$X$ is a fourfold or sixfold which is not ``Hodge-theoretically special'' in the following sense. 

First, we note that if $n$ denotes the dimension of $X$, 
then by Lefschetz theorems (see~\cite[Proposition~3.4(b)]{debarre-kuznetsov-periods}) 
the Gushel map $f\colon X \to \bG$ induces an injection 
\begin{equation*}
\rH^{n}(\bG, \bZ) \hookrightarrow \rH^{n}(X, \bZ). 
\end{equation*} 
If $n$ is odd, then $\rH^n(\bG, \bZ)$ simply vanishes.
But if $n=4$ or $6$, then $\rH^n(\bG, \bZ) = \bZ^2$ is generated 
by Schubert cycles, and the \emph{vanishing cohomology} 
$\rH^n_{\van}(X, \bZ)$ is defined as the orthogonal to 
$\rH^{n}(\bG, \bZ) \subset \rH^{n}(X, \bZ)$ with respect to the 
intersection form. 

\begin{definition}[\cite{DIM4fold}]
\label{definition-Hodge-special}
Let $X$ be a GM variety of dimension $n = 4$ or $6$. 
Then $X$ is \emph{Hodge-special} if
\begin{equation*}
\rH^{\frac{n}{2}, \frac{n}{2}}(X) \cap \rH^n_{\mathrm{van}}(X,\bQ) \ne 0. 
\end{equation*}
\end{definition}

\begin{lemma}[\cite{DIM4fold}]
\label{lemma-not-Hodge-special}
If $X$ is a very general GM fourfold or sixfold, then 
$X$ is not \mbox{Hodge-special}. 
\end{lemma}

\begin{remark}
\label{remark:very-general}
\emph{Very general} here means that the moduli point $[X] \in \cM_n(\bC)$ 
lies in the complement of countably many proper closed substacks of $\cM_n$, 
where $n = \dim(X)$ and $\cM_n$ is the moduli stack of $n$-dimensional GM varieties  
discussed in Appendix~\ref{appendix-moduli}. 
\end{remark}

\begin{proof}
In the fourfold case, this is \cite[Corollary~4.6]{DIM4fold}. 
The main point of the proof is the computation that the local period 
map for GM fourfolds is a submersion. 
The sixfold case can be proved by the same argument.
\end{proof}

\begin{remark}
\label{remark:periods-isomorphism}
Lemma~\ref{lemma-not-Hodge-special} can also be proved by combining the the description of the moduli of GM varieties in terms of EPW sextics 
(see Remark~\ref{remark:moduli-period}) with \cite[Theorem 5.1]{debarre-kuznetsov-periods}. 
\end{remark}

\begin{proposition}
\label{proposition-gg-AX}
Let $X$ be a GM variety of dimension $n \ge 3$. 
If $n$ is even assume also that~$X$ is not Hodge-special.
Then $\rK_{0}(\cA_{X})_{\num} \simeq \bZ^2$.
\end{proposition}

\begin{proof}
The proof is similar to that of Proposition~\ref{proposition-HH_*}. 
First, note that by Proposition~\ref{proposition-dbx} there is a semiorthogonal decomposition of 
$\Db(X)$ with $\cA_X$ and $2(n-2)$ exceptional objects as components. 
Since the category generated by an exceptional object is equivalent 
to the derived category of a point, both its usual and numerical Grothendieck group is $\bZ$. 
Hence by additivity, 
\begin{equation*}
\rK_0(X)_{\num} \cong \rK_0(\cA_X)_{\num} \oplus 
\bZ^{2(n-2)}. 
\end{equation*}
On the other hand, 
$\rK_0(X)_{\num} \otimes \bQ \cong \CH(X)_{\num} \otimes \bQ$. 
But under our assumptions on $X$, the rational Hodge classes on $X$
are spanned by the restrictions of Schubert cycles on $\bG$. 
In particular, the Hodge conjecture holds for $X$. So numerical equivalence coincides with homological equivalence, and
\begin{equation*}
\textstyle{\CH(X)_{\num}\otimes \bQ 
\cong \bigoplus_{k} \rH^{k,k}(X, \bQ) }
\end{equation*} 
where $\rH^{k,k}(X, \bQ) = \rH^{k,k}(X) \cap \rH^{2k}(X,\bQ)$. 
Thus using the Hodge diamond of $X$ (recorded in the proof of 
Proposition~\ref{proposition-HH_*}) 
and the assumption that $X$ is not Hodge-special if 
$n$ is even, we find  
\begin{equation*}
\dim (\rK_0(X)_{\num} \otimes \bQ) = 2n - 2.
\end{equation*} 
Combined with the above, this shows the rank of $\rK_0(\cA_X)_{\num}$ is $2$. 
Since $\rK_0(\cA_X)_{\num}$ is torsion free, we conclude $\rK_0(\cA_X)_{\num} \cong \bZ^2$. 
\end{proof}

\begin{remark}
\label{remark-K0-hodge}
Let $X$ be a GM variety of dimension $n=4$ or $6$. 
The proof of the proposition shows that
\begin{equation*}
\rank(\rK_{0}(\cA_{X})_{\num})= \dim_{\bQ} \rH^{\frac{n}{2}, \frac{n}{2}}(X, \bQ)
\end{equation*}
if the Hodge conjecture holds for $X$. 
The Hodge conjecture holds for any uniruled smooth projective fourfold~\cite{conte-murre}, so for $n=4$ the above equality is unconditional. 
If $n=6$ the Hodge conjecture can be proved using the correspondences studied in \cite{debarre-kuznetsov-periods}, but we do not discuss the details here. 
\end{remark}

\begin{lemma}
\label{lemma-euler-form}
Let $X$ be a GM variety as in Proposition~\textup{\ref{proposition-gg-AX}}. 
Then in a suitable basis, the Euler form on 
$\rK_0(\cA_X)_{\num} = \bZ^2$ is given by
\begin{equation*}
\begin{pmatrix}
-1 & 0 \\
0 & -1
\end{pmatrix} \hspace{2mm} \text{if $n = 3$}, \qquad
\begin{pmatrix}
-2 & 0 \\
0 & -2
\end{pmatrix}  \hspace{2mm} \text{if $n = 4$}. 
\end{equation*} 
\end{lemma}

\begin{remark}
The duality conjecture (Conjecture~\ref{conjecture:duality}) 
implies that if $X$ is as in Proposition~\ref{proposition-gg-AX}, then 
for $n = 5$ or $6$ the lattice $\rK_0(\cA_X)_{\num} = \bZ^2$ is isomorphic 
to the lattice described in Lemma~\ref{lemma-euler-form} for $n = 3$ or $4$, 
respectively.
\end{remark}

\begin{proof}
For $n = 3$, this is shown in the proof of \cite[Proposition~3.9]{kuznetsov-fano-3-folds}. 

For $n = 4$, we sketch the proof. 
First, note that any GM variety contains a line, since by taking {a hyperplane section}
we reduce to the case of dimension $3$, where the result is well-known. 
Let $P \in X$ be a point, 
$L \subset X$ be a line, 
$\Sigma$ be the zero locus of a regular section of~$\cU_X^{\vee}$,  
$S$ be a complete intersection of two hyperplanes in $X$, and 
$H$ be a hyperplane section of $X$.  
The key claim is that 
\begin{equation}
\label{basis-gg}
\rK_0(X)_{\num} = \bZ \langle [\cO_P], [\cO_{L}] , [\cO_{\Sigma}], [\cO_{S}], [\cO_{H}], [\cO_X] \rangle , 
\end{equation}
i.e. the structure sheaves of these subvarieties give an integral basis of $\rK_0(X)_{\num}$. 
Once this is known, as in the proof of~\cite[Proposition~3.9]{kuznetsov-fano-3-folds}, 
the lemma reduces to a (tedious) computation, which we omit. 

Using~\cite[Remark~5.9]{kuznetsov-fano-3-folds} 
it is easy to see $X$ is AK-compatible in the sense 
of~\cite[Definition~5.1]{kuznetsov-fano-3-folds}, 
hence to prove the claim it is enough to show that  
\begin{equation*}
\CH(X)_{\num} = \bZ \langle [P], [L] , [\Sigma], [S], [H], 1 \rangle.
\end{equation*}
Clearly, this is equivalent to  
$\CH^2(X)_{\num} = \bZ \langle [\Sigma], [S] \rangle$. 
But $\CH^2(X)_{\num}$ coincides with the group 
$\CH^2(X)_{\hom} \subset \rH^4(X, \bZ)$ of $2$-cycles 
modulo homological equivalence (see the proof of Proposition~\ref{proposition-gg-AX}), 
and $\bZ\langle [\Sigma], [S] \rangle$ is the image of the inclusion 
$\rH^4(\bG, \bZ) \hookrightarrow \CH^2(X)_{\hom}$. Hence 
it suffices to show the cokernel of this inclusion is torsion free. 
Even better, the cokernel of
\begin{equation*}
\rH^4(\bG,\bZ) \hookrightarrow \rH^4(X, \bZ)
\end{equation*}
is torsion free. Indeed, we may assume $X$ is ordinary, 
and then the statement holds by the proof of the Lefschetz 
hyperplane theorem, see~\cite[Example~3.1.18]{lazarsfeld}.  
\end{proof}

\subsection{Geometricity of GM categories} 
\label{subsection-geometricity}
Now we consider the question of whether $\cA_X$ is 
equivalent to the derived category of a variety. 
The following two results show that in almost all cases, 
the answer is negative. 
In \S\ref{subsection-rationality-conjecture} we will discuss a related 
conjecture about the rationality of GM fourfolds. 

\begin{proposition}
\label{proposition-AX-geometric}
Let $X$ be a GM variety of dimension $n$. 
\begin{enumerate}
\item \label{AX-geometric-even}
If $n$ is even and $S$ is a variety such that $\cA_X \simeq \Db(S)$, then $S$ is a K$3$ surface. 
\item  \label{AX-geometric-odd} 
If $n$ is odd, then $\cA_X$ is not equivalent to the derived category of any variety. 
\item \label{AX-geometric-even-very-general}
If $n = 4$ or $n = 6$ and $X$ is not Hodge-special \textup{(}in particular, if $X$ is very general\textup{)}, 
then $\cA_X$ is not equivalent to the derived category of any variety.
\end{enumerate}
\end{proposition}

\begin{proof}
Suppose $S$ is a variety such that $\cA_X \simeq \Db(S)$. 
Then $S$ is smooth by \cite[Lemma~D.22]{kuznetsov2006hyperplane}, 
and proper by \cite[Proposition~3.30]{orlov-ncschemes}. 
In particular, $\Db(S)$ has a Serre functor given by 
\begin{equation*}
\rS_{\Db(S)}(\cF) = \cF \otimes \omega_S [\dim(S)] , 
\end{equation*}
which is unique up to isomorphism. 
Thus by Proposition~\ref{proposition-serre-functor}, $S$ is a surface with trivial (if~$n$ is even) or $2$-torsion (if $n$ is odd) canonical class. 
Hence $S$ is a K3, Enriques, abelian, or bielliptic surface. 
Using the HKR isomorphism and the Hodge diamonds of such surfaces, we find  
\begin{equation*}
\HH_{\bullet}(S) = 
\arraycolsep=1.4pt
\left\{\begin{array}{cccccccccl}
&& \bC[2] & \oplus & \bC^{22} & \oplus & \bC[-2] && & \qquad\text{if $S$ is K3}, \\
&&&& \bC^{12} &&&& & \qquad\text{if $S$ is Enriques}, \\
\bC[2] & \oplus & \bC^4[1] & \oplus & \bC^{6} & \oplus & \bC^4[-1] & \oplus & \bC[-2] & \qquad\text{if $S$ is abelian},  \\
&& \bC^2[1] & \oplus & \bC^{4} & \oplus & \bC^2[-1] && & \qquad\text{if $S$ is bielliptic}. 
\end{array}
\right.
\end{equation*} 
Now parts \eqref{AX-geometric-even} and \eqref{AX-geometric-odd} follow by comparing with $\HH_{\bullet}(\cA_X)$ as given by Proposition~\ref{proposition-HH_*}. 
For \eqref{AX-geometric-even-very-general} note that if $\cA_X \simeq \Db(S)$, then $\rK_{0}(\cA_{X})_{\num} \cong \rK_{0}(S)_{\num}$. 
But on a projective surface powers of the hyperplane class give $3$ independent elements in~$\CH(S)_{\num} \otimes \bQ \cong \rK_{0}(S)_{\num} \otimes \bQ$. 
Hence by Proposition \ref{proposition-gg-AX}, $X$ is Hodge-special. 
\end{proof}

\subsection{Self-duality of GM categories}
\label{subsection:self-duality}

The derived category of a smooth proper variety~$X$ is \emph{self-dual}: 
if $\Db(X)^{\opp}$ denotes the opposite category of $\Db(X)$ (note that this has nothing to do with the opposite GM variety), 
there is an equivalence $\Db(X) \simeq \Db(X)^{\opp}$ given by 
the dualization functor $\cF \mapsto \RCHom(\cF,\cO_X)$.
In general, this self-duality property is not inherited by semiorthogonal components of~$\Db(X)$. 
Nonetheless, we show below that all GM categories are self-dual,    
which can be thought of as a weak geometricity property. 

For the proof, we recall some facts about mutation functors 
(see \cite{bondal}, \cite{bondal-kapranov} for more details). 
For any admissible subcategory 
$\cA \subset \cT$ of a triangulated category, there are associated 
\emph{left} and \emph{right mutation functors} 
$\rL_{\cA} \colon \cT \to \cT$ and $\rR_{\cA} \colon \cT \to \cT$. 
These functors annihilate~$\cA$, and their restrictions
$\rL_{\cA}|_{{}^{\perp}\cA} \colon {}^{\perp}\cA \to \cA^{\perp}$ and
$\rR_{\cA}|_{\cA^{\perp}}  \colon \cA^{\perp} \to {}^{\perp}\cA$
are mutually inverse equivalences \cite[Lemma~1.9]{bondal-kapranov}. 
If $\cT = \langle \cA_1, \dots, \cA_n \rangle$ is a semiorthogonal 
decomposition with admissible components, then for 
$1 \leq i \leq n-1$ there are semiorthogonal decompositions 
\begin{align*}
\cT &= \langle \cA_1,\dots,\cA_{i-1},\rL_{\cA_i}(\cA_{i+1}),\cA_i,\cA_{i+2},\dots,\cA_n \rangle , \\
\cT &= \langle \cA_1,\dots,\cA_{i-1},\cA_{i+1},\rR_{\cA_{i+1}}(\cA_i),\cA_{i+2},\dots,\cA_n \rangle, 
\end{align*}
and equivalences 
\begin{equation}
\label{eq:mutation-components-equivalences}
\rL_{\cA_i}(\cA_{i+1}) \simeq \cA_{i+1}
\qquad\text{and}\qquad 
\rR_{\cA_{i+1}}(\cA_i) \simeq \cA_i
\end{equation}
induced by the mutation functors 
$\rL_{\cA_i} \colon \cT \to \cT$ and $\rR_{\cA_{i+1}} \colon \cT \to \cT$. 
When $\cT$ admits a Serre functor $\rS_\cT$, the effect of mutating $\cA_n$ or $\cA_1$ 
to the opposite side of the semiorthogonal decomposition of $\cT$ can be described 
as follows \cite[Proposition~3.6]{bondal-kapranov}: 
\begin{equation}
\label{eq:mutation-serre-functor}
\cT = \langle \rS_\cT(\cA_n), \cA_1, \dots, \cA_{n-1} \rangle
\qquad\text{and}\qquad 
\cT = \langle \cA_2, \dots, \cA_{n}, \rS_\cT^{-1}(\cA_1) \rangle . 
\end{equation}
That is, $\rL_{\langle \cA_1, \dots, \cA_{n-1} \rangle}(\cA_n) = \rS_\cT(\cA_n)$ and $\rR_{\langle \cA_2, \dots, \cA_n \rangle}(\cA_1) = \rS_\cT^{-1}(\cA_1)$. 

\begin{lemma}
\label{lemma:ax-selfdual}
For any GM variety $X$ the corresponding GM category $\cA_X$ is self-dual, i.e.
\begin{equation*}
\cA_X \simeq \cA_X^\opp.
\end{equation*}
\end{lemma}

\begin{proof}
If $\dim(X) = 2$ then $\cA_X = \Db(X)$, so the result holds by self-duality of $\Db(X)$. 
Now assume $\dim(X) \geq 3$. 
Applying the dualization functor $\cF \mapsto \RCHom(\cF,\cO_X)$ 
to the semiorthogonal decomposition~\eqref{eq:dbx-detailed},    
we obtain a new semiorthogonal decomposition
\begin{equation}
\label{eq:dbx-dualized}
\Db(X) = \langle \cU_X(-(n-3)), \cO_X(-(n-3)), \dots, \cU_X, \cO_X, \cA'_X \rangle
\end{equation} 
and an equivalence $\cA'_X \simeq \cA_X^\opp$.
It remains to show 
\begin{equation}
\label{eq:dual-gm-category}
\cA'_X \simeq \cA_X.
\end{equation} 
We mutate the subcategory $\langle \cU_X(-(n-3)), \cO_X(-(n-3)), \dots, \cU_X \rangle$ to the 
far right side of~\eqref{eq:dbx-dualized}. 
By \eqref{eq:mutation-serre-functor}, the formula~\eqref{eq:serre-x} 
for the Serre functor of $\Db(X)$, and the formula~\eqref{eq:omega-x} for 
$-K_X$, the result is 
\begin{equation*}
\Db(X) = \langle \cO_X, \cA'_X, \cU_X(1), \cO_X(1), \dots, \cU_X(n-2) \rangle.
\end{equation*}
Using the isomorphism $\cU_X(1) \cong \cU_X^\vee$ and 
comparing this decomposition with~\eqref{eq:dbx-detailed}, 
we deduce that 
$\cA_X = \rL_{\cO_X}(\cA'_X)$. 
Hence $\cA_X \simeq \cA'_X$ by~\eqref{eq:mutation-components-equivalences}. 
\end{proof}

\begin{remark}
A similar argument shows that the K3 category associated to a cubic fourfold (as defined by~\eqref{sod-cubic} below) is self-dual. 
\end{remark}

\section{Conjectures on duality and rationality} 
\label{section-conjectures} 

In this section, we propose two conjectures related to the variation of GM categories $\cA_X$ as~$X$ varies in moduli. 
We begin by briefly recalling a description of the moduli of GM varieties in terms of 
EPW sextics from~\cite[\S3]{debarre-kuznetsov} (see Appendix~\ref{appendix-moduli} for 
some basic results about the moduli stack of GM varieties). 
Using this, we formulate a duality conjecture (Conjecture~\ref{conjecture:duality}), which 
in particular implies that $\cA_X$ is constant in families of GM varieties with the same 
assoicated EPW sextic. 
Next we discuss the rationality problem for GM varieties in terms of GM categories. 
This problem is most interesting for GM fourfolds, where by analogy with cubic fourfolds 
we conjecture that the GM category of a rational GM fourfold is equivalent to the 
derived category of a K3 surface (Conjecture~\ref{conjecture-rational-fourfold}). 

\subsection{EPW sextics and moduli of GM varieties} 
\label{subsection-EPW}

Let $V_6$ be a $6$-dimensional vector space. Its exterior power $\wedge^3V_6$ has a natural 
$\det(V_6)$-valued symplectic form, given by wedge product.
For any Lagrangian subspace $\sA \subset \wedge^3 V_6$, we consider the following 
stratification of $\bP(V_6)$: 
\begin{equation*}
\sY_\sA^{\geq k}  = 
\{ v \in \bP(V_6) \mid \dim(\sA \cap (v \wedge (\wedge^2 V_6))) \geq k \} \subset \bP(V_6).
\end{equation*}
We write $\sY_{\sA}^k$ for the complement of $\sY_{\sA}^{\geq k+1}$ in $\sY_{\sA}^{\geq k}$, and $\sY_\sA$ for $\sY_\sA^{\ge 1}$. 
The variety $\sY_{\sA}$ is called an \emph{EPW sextic} 
(for Eisenbud, Popescu, and Walter, who first defined it), 
and the sequence~$\sY_{\sA}^k$ is called the \emph{EPW stratification}. 

We say $\sA$ has \emph{no decomposable vectors} if $\bP(\sA)$ does not intersect
$\G(3,V_6) \subset \bP(\wedge^3V_6)$. 
O'Grady \cite{ogrady2006irreducible,ogrady2008dual,ogrady2011moduli,ogrady2012taxonomy,ogrady2013double,ogrady2015periods}
extensively investigated the geometry of EPW sextics, and proved in particular  
that (see also~\cite[Theorem B.2]{debarre-kuznetsov}) if $\sA$ has no decomposable vectors, then:
\begin{itemize}
\item 
$\sY_\sA = \sY_\sA^{\ge 1}$ is a normal irreducible sextic hypersurface, smooth along $\sY_\sA^1$;
\item 
$\sY_\sA^{\ge 2} = \operatorname{Sing}(\sY_\sA)$ is a normal irreducible surface of degree 40, smooth along $\sY_\sA^2$;
\item 
$\sY_\sA^3 = \operatorname{Sing}(\sY_\sA^{\ge 2})$ is finite and reduced, 
and for general $\sA$ is empty;
\item 
$\sY_\sA^{\ge 4} = \varnothing$.
\end{itemize}

For any Lagrangian subspace $\sA \subset \wedge^3V_6$, its orthogonal 
$\sA^\perp = \ker(\wedge^3V_6^\vee \to \sA^{\vee}) \subset \wedge^3 V_6^{\vee}$ is also Lagrangian, 
and $\sA$ has no decomposable vectors if and only if the same is true for $\sA^\perp$.
In particular, $\sA^{\perp}$ gives rise to an EPW sequence of subvarieties of $\bP(V_6^{\vee})$, 
which can be written in terms of $\sA$ as follows: 
\begin{equation*}
\sY_{\sA^\perp}^{\ge k} = \{ V_5 \in \bP(V_6^\vee) \mid \dim(\sA \cap \wedge^3V_5) {}\ge{} k \} \subset \bP(V_6^\vee).
\end{equation*}
This stratification has the same properties as the stratification $\sY_\sA^{\ge k}$.
By O'Grady's work $\sY_{\sA^\perp}$ is projectively dual to $\sY_{\sA}$, 
and for this reason is called the \emph{dual EPW sextic} to $\sY_{\sA}$. 
We note that $\sY_{\sA^\perp}$ is not isomorphic to 
$\sY_\sA$ for general $\sA$ (see \cite[Theorem~1.1]{ogrady2008dual}).

One of the main results of~\cite{debarre-kuznetsov} is the following description of the set 
of all isomorphism classes of smooth ordinary GM varieties.
If $X \subset \bP(W)$ is a GM variety, then the space of quadrics in $\bP(W)$ containing 
$X$ is $6$-dimensional vector space \cite[Theorem~2.3]{debarre-kuznetsov}, which we denote by $V_6(X)$. 
The space of Pl\"ucker quadrics defining the Grassmannian $\bG = \G(2,V_5(X))$ is 
canonically identified with $V_5(X)$, so since $X \subset \Cone(\bG)$ we have an embedding 
\begin{equation*}
V_5(X) \subset V_6(X).
\end{equation*}
The hyperplane $V_5(X)$ is called \emph{the Pl\"ucker hyperplane of $X$} and the corresponding point
\begin{equation*}
\bp_X \in \bP(V_6(X)^\vee)
\end{equation*}
is called \emph{the Pl\"ucker point of $X$}.
Furthermore, in~\cite[Theorem~3.10]{debarre-kuznetsov} it is shown that there is a natural Lagrangian subspace
\begin{equation*}
\sA(X) \subset \wedge^3V_6(X) 
\end{equation*}
associated to $X$. 
If $X^\op$ is the opposite variety of $X$ as defined by~\eqref{opp-variety}, then $\sA(X^\opp) = \sA(X)$ and $\bp_{X^\op} = \bp_X$.

\begin{theorem}[\protect{\cite[Theorem~3.10]{debarre-kuznetsov}}]
\label{theorem:dk:moduli}
{For any $n \ge 2$ the maps $X \to X^\opp$ and $X \mapsto (\sA(X),\bp_X)$ define bijections between}
\begin{enumerate}
\item 
the set of ordinary GM varieties $X$ of dimension $n \geq 2$ 
whose Grassmannian hull $M_X$ is smooth, up to isomorphism, 
\item 
the set of special GM varieties of dimension ${n+1} \ge 3$, up to isomorphism, and 
\item 
the set of pairs $(\sA,\bp)$, where $\sA \subset \wedge^3V_6$ is a Lagrangian subspace with 
no decomposable vectors and $\bp \in \sY_{\sA^\perp}^{5-n}$, 
up to the action of $\PGL(V_6)$.
\end{enumerate}
\end{theorem}

Note that by Lemma~\ref{lemma:gr-hull}, $M_X$ is automatically smooth for ordinary GM varieties of dimension $n \geq 3$.

\begin{remark}
\label{remark-surface-lagrangian}
To include all GM surfaces into the above bijection, we must allow a more general 
class of Lagrangian subspaces in Theorem~\ref{theorem:dk:moduli}, 
namely those that contain finitely many decomposable vectors, 
cf.~\cite[Theorem~3.16 and Remark~3.17]{debarre-kuznetsov}. 
\end{remark}

\begin{remark}
\label{remark:moduli-period}
Theorem~\ref{theorem:dk:moduli} 
suggests there is a morphism from the moduli stack $\cM_n$ of \mbox{$n$-dimensional} GM varieties (see Appendix~\ref{appendix-moduli})
to the quotient stack $\LGr(\wedge^3V_6)/\PGL(V_6)$ (where $\LGr(\wedge^3V_6)$ is the Lagrangian Grassmannian) 
given by $X \mapsto \sA(X)$ at the level of points, whose fiber over a point $\sA$ is the union of two EPW strata $\sY_{\sA^\perp}^{5-n} \sqcup \sY_{\sA^\perp}^{6-n}$,  
modulo the action of the stabilizer of $\sA$ in $\PGL(V_6)$. 
This morphism will be discussed in detail in~\cite{debarre-kuznetsov-moduli}. 
Let us simply note that it gives a geometric way 
to compute $\dim \cM_n$ (cf.~Proposition~\ref{proposition-moduli}). 
Namely, the quotient stack $\LGr(\wedge^3V_6)/\PGL(V_6)$ has dimension $20$, and the fibers 
of the supposed morphism have dimension $5$, $5$, $4$, or $2$ 
for $n = 6, 5, 4$, or $3$, respectively. 
Finally, for~$n = 2$ the morphism is no longer dominant, as its 
image is the divisor of those $\sA$ such that~$\sY_{\sA^\perp}^3 \ne \varnothing$, and its fibers are finite.
\end{remark}

The above discussion shows the utility of the EPW stratification of $\bP(V_6^{\vee})$ from the point of view of moduli. 
The following proposition gives a geometric interpretation of the EPW stratification of $\bP(V_6)$, 
which will be essential later.  

As mentioned before, the quadric $Q$ defining $X$ in~\eqref{x:gm:general} is not unique; 
such quadrics are parameterized by the affine space $\bP(V_6(X)) \setminus \bP(V_5(X))$ of non-Pl\"ucker quadrics.
In other words, a quadric $Q$ defining $X$ in~\eqref{x:gm:general} corresponds to a \emph{quadric point} 
\begin{equation*}
 \bq \in \bP(V_6(X))
\end{equation*}
such that $(\bq,\bp_X)$ does not lie on the incidence divisor in $\bP(V_6(X)) \times \bP(V_6(X)^\vee)$. 

\begin{proposition}[\protect{\cite[Proposition~3.13(b)]{debarre-kuznetsov}}]
\label{proposition:dk:quadrics} 
Let $X$ be a GM variety. 
Under the identification of the affine space $\bP(V_6(X)) \setminus \bP(V_5(X))$ with the 
space of non-Pl\"ucker quadrics containing~$X$, 
the stratum 
\begin{equation*}
\sY_{\sA(X)}^k \cap (\bP(V_6(X)) \setminus \bP(V_5(X)))
\end{equation*}
corresponds to the quadrics $Q$ such that $\dim(\ker(Q)) = k$.
\end{proposition}

The symmetry between the Pl\"ucker point $\bp_X$
and the quadric point $\bq$ is the basis 
for the duality of GM varieties, discussed below.

\subsection{The duality conjecture}
\label{subsection:duality}
The following definition extends~\cite[Definitions~3.22 and~3.26]{debarre-kuznetsov}.

\begin{definition}\label{definition:duality}
Let $X_1$ and $X_2$ be GM varieties. 
\begin{enumerate}
\item If there exists an isomorphism $V_6(X_1) \cong V_6(X_2)$ identifying $\sA(X_1) \subset \wedge^3V_6(X_1)$ with~$\sA(X_2) \subset \wedge^3V_6(X_2)$, then we say:  
\begin{itemize}
 \item $X_1$ and $X_2$ are \emph{period partners} if $\dim(X_1) = \dim(X_2)$, and 
 \item $X_1$ and $X_2$ are \emph{generalized partners} if $\dim(X_1) \equiv \dim(X_2) \pmod 2$.
\end{itemize}

\item 
If there exists an isomorphism $V_6(X_1) \cong V_6(X_2)^{\vee}$ identifying $\sA(X_1) \subset \wedge^3V_6(X_1)$ with~$\sA(X_2)^\perp \subset \wedge^3V_6(X_2)^\vee$, then we say:  
\begin{itemize}
 \item $X_1$ and $X_2$ are \emph{dual} if $\dim(X_1) = \dim(X_2)$, and 
 \item $X_1$ and $X_2$ are \emph{generalized dual} if $\dim(X_1) \equiv \dim(X_2) \pmod 2$.
\end{itemize}
\end{enumerate}
\end{definition}

\begin{remark}
\label{remark-duality}
If $X$ is a GM variety, then either $\sA(X)$ does or does 
not contain decomposable vectors, and these two cases are preserved by generalized partnership and duality. 
The first case happens only when $X$ is an ordinary surface with singular Grassmannian hull 
or $X$ is a special surface, see~\cite[Theorem~3.16 and Remark~3.17]{debarre-kuznetsov}. 
In this paper, we focus on the case where $\sA(X)$ does not contain decomposable vectors.
\end{remark}

One of the main results of~\cite[\S4]{debarre-kuznetsov} is 
that period partners or dual GM varieties of dimension at least $3$ are birational. 
Our motivation for defining generalized partners and duals is the following conjecture.  

\begin{conjecture}\label{conjecture:duality}
Let $X_1$ and $X_2$ be GM varieties such that the subspaces $\sA(X_1)$ and $\sA(X_2)$ do not contain decomposable vectors,
and let $\cA_{X_1}$ and $\cA_{X_2}$ be their GM categories.
\begin{enumerate}
\item \label{conjecture-partners} 
If $X_1$ and $X_2$ are generalized partners, there is an equivalence $\cA_{X_1} \simeq \cA_{X_2}$.
\item \label{conjecture-duality}
If $X_1$ and $X_2$ are generalized duals, 
there is an equivalence $\cA_{X_1} \simeq \cA_{X_2}$.
\end{enumerate}
\end{conjecture}

By Proposition~\ref{proposition-serre-functor}, GM varieties with equivalent GM categories must have dimensions of the same parity, 
which explains the parity condition in Definition~\ref{definition:duality}.
We note that part \eqref{conjecture-partners} of the conjecture follows from part \eqref{conjecture-duality}, 
since by Definition~\ref{definition:duality} and Theorem~\ref{theorem:dk:moduli}
generalized period partners have a common generalized dual GM variety.
For this reason, we refer to Conjecture~\ref{conjecture:duality} 
as the \emph{duality conjecture}.

As evidence for the duality conjecture, we prove in \S\ref{section-associated-K3} 
the special case where $X_1$ is an ordinary GM fourfold and $X_2$ is a (suitably generic) generalized dual GM surface. 
In fact, the approach of \S\ref{section-associated-K3} can be used to attack 
the full conjecture, but is quite unwieldy to carry out in the general case.
In forthcoming work, we establish the general case as a consequence of a theory of ``categorical joins'' \cite{joins}.
This approach is based on the observation from~\cite[Proposition~3.28]{debarre-kuznetsov} that 
duality of ordinary GM varieties can be interpreted in terms of projective duality of quadrics 
(see also~\S\ref{subsection-setup-associated-K3}). 
We show that this extends to generalized duality by replacing classical projective duality with 
homological projective duality. 

In the rest of this subsection we discuss some consequences of the duality conjecture.
We start by describing all generalized duals and partners of a given GM variety. 

\begin{lemma}
\label{lemma:moduli-gm-duals-partners}
Let $X$ be an $n$-dimensional GM variety, and assume $\sA(X)$ has no decomposable vectors.
Then any quadric point $\bq \in \bP(V_6(X))$ corresponds to a generalized dual~$X_{\bq}^{\vee}$ of $X$. 
If~$\bq$ lies in the stratum $\sY_{\sA(X)}^k$ for some $k$, we have:
\begin{itemize}
\item 
If \mbox{$5-k \equiv n \pmod 2$}, then $X_{\bq}^{\vee}$ is an ordinary GM variety of dimension $5-k$.
\item 
If $6-k \equiv n \pmod 2$, then $X_{\bq}^{\vee}$ is a special GM variety of dimension $6-k$.
\end{itemize}
Similarly, any point $\bp \in \bP(V_6(X)^\vee)$ corresponds to a generalized partner $X_\bp$ of $X$.

Conversely, any generalized dual of~$X$ arises as~$X_{\bq}^{\vee}$ for some $\bq \in \bP(V_6(X))$
and any generalized partner of~$X$ arises as~$X_{\bp}$ for some $\bp \in \bP(V_6(X)^\vee)$.
\end{lemma}
\begin{proof}
The variety $X_{\bq}^{\vee}$ corresponding to a quadric point $\bq \in \bP(V_6)$ is just the ordinary GM variety of dimension $5-k$ or the special GM variety of dimension $6-k$ 
associated by Theorem~\ref{theorem:dk:moduli} to the pair $(\sA(X)^\perp,\bq)$ (with $V_6 = V_6(X)^\vee$). 
It also follows from Theorem~\ref{theorem:dk:moduli} that any generalized dual of $X$ arises in this way.

The same argument also works for generalized partners. 
\end{proof}

The argument of Lemma~\ref{lemma:moduli-gm-duals-partners} shows that 
the set of isomorphism classes of generalized duals of~$X$ can be identified with the quotient of $\bP(V_6(X))$
by the action of the stabilizer of~$\sA(X)$ in $\PGL(V_6(X))$. 
Analogously, the isomorphism classes of generalized partners of $X$ are parameterized by a quotient of $\bP(V_6(X)^\vee)$ by the same group. 

Let us list more explicitly the type of $X_{\bq}^{\vee}$ according to the stratum 
$\sY_{\sA(X)}^k$ of $\bq$ and the parity of $n$: 
\begin{center}
{\tabulinesep=1.2mm
\begin{tabu}{|c|c|c|}
\hline 
$k$ & $X_{\bq}^{\vee}$ for $n$ even & $X_{\bq}^{\vee}$ for $n$ odd \\ 
\hline 
$0$ & special sixfold  & ordinary fivefold \\
\hline 
$1$ & ordinary fourfold & special fivefold \\
\hline 
$2$ & special fourfold & ordinary threefold \\ 
\hline 
$3$ & ordinary surface & special threefold \\
\hline 
\end{tabu}}
\end{center}
Recall that the stratum $\sY_{\sA(X)}^k$ is always nonempty for 
$k = 0,1,2$, generically empty for $k = 3$, and always empty for $k \geq 4$ 
(under our assumption that $\sA(X)$ contains no decomposable vectors). 
In fact, the condition that $\sY_{\sA(X)}^3$ is nonempty is divisorial 
in $\cM_n$ (see Remark~\ref{remark-EWP3-divisor}). 
In the same way, one can describe the types of generalized partners $X_\bp$ of $X$ depending on the stratum $\sY_{\sA(X)^\perp}^k$ of $\bp$ and the parity of $n$.

Conjecture~\ref{conjecture:duality} says there are equivalences 
\begin{equation*}
\cA_{X} \simeq \cA_{X_{\bp}} \simeq \cA_{X_\bq^\vee}
\end{equation*}
for every $\bp \in \bP(V_6(X)^\vee)$ and every $\bq \in \bP(V_6(X))$. 
In particular, it predicts that often GM categories are equivalent to 
those of lower-dimensional GM varieties, namely that:
\begin{enumerate}
\item \label{duality-ex-1} If $X$ is a sixfold, then its GM category is equivalent to a fourfold's GM category.  
\item \label{duality-ex-2} If $X$ is a fivefold, then its GM category is equivalent to a threefold's GM category. 
\item \label{duality-ex-3} If $X$ is a fourfold such that $\sY_{\sA(X)^\perp}^{3} \neq \varnothing$ or $\sY_{\sA(X)}^{3} \neq \varnothing$, then 
its GM category is equivalent to the derived category of a GM surface. 
\end{enumerate}

As mentioned above, in \S\ref{section-associated-K3} we prove  \eqref{duality-ex-3} in case $\sY_{\sA(X)}^{3} \neq \varnothing$ and 
an additional genericity assumption holds, namely $\sY_{\sA(X)}^3 \not\subset \bP(V_5(X))$.

\begin{remark}
Using Theorem~\ref{theorem:dk:moduli},
it is easy to see that to prove the duality conjecture in full generality, 
it is enough to prove $\cA_{X} \simeq \cA_{X_{\bq}^{\vee}}$ for all $X$ and $\bq \in \bP(V_6(X)) \setminus \bP(V_5(X))$. 
A similar reduction was used in~\cite[\S4]{debarre-kuznetsov} to prove birationality of period partners and 
of dual GM varieties.
\end{remark}

\begin{remark}
A GM variety $X$ as in \eqref{duality-ex-1}--\eqref{duality-ex-3} above 
is rational (see the discussion below and Lemma~\ref{lemma-rational-GM}). 
It seems likely that for such an $X$ there is a rationality construction that involves 
a blowup of a generalized partner or dual variety of dimension $2$ less, and gives rise to an 
equivalence of GM categories. Our approach to \eqref{duality-ex-3} in 
\S\ref{section-associated-K3} takes a completely different route. 
\end{remark}

\subsection{Rationality conjectures} 
\label{subsection-rationality-conjecture}
Let us recall what is known about rationality of GM varieties. 
A general GM threefold is irrational by~\cite[Theorem~5.6]{beauville1977}, while 
every GM fivefold or sixfold is rational by~\cite[Proposition~4.2]{debarre-kuznetsov} 
(for a \emph{general} GM fivefold or sixfold this was already known to Roth). 
The intermediate case of GM fourfolds is more mysterious, and closely parallels the 
situation for cubic fourfolds: some rational examples are known~\cite{DIM4fold}, but 
while a very general GM fourfold is expected to be irrational, it has not been proven 
that a single GM fourfold is irrational. 
Below, we analyze this state of affairs from the point of view of derived categories.  

Following \cite[\S3.3]{kuznetsov2015rationality}, we use the following terminology: 
\begin{itemize}  
\item For a triangulated category $\cA$, the \emph{geometric dimension} $\gdim(\cA)$ is defined as the minimal 
integer $m$ such that there exists an $m$-dimensional connected smooth
projective variety $M$ and an admissible embedding $\cA \hookrightarrow \Db(M)$. 
\item If $Y$ is a smooth projective variety and $\Db(Y) =  \langle \cA_1, \dots, \cA_m \rangle$ is a maximal semiorthogonal decomposition (i.e. the components are indecomposable), then $\cA_i$ is 
called a \emph{Griffiths component} if $\gdim(\cA_i) \ge \dim Y - 1$. 
\end{itemize}
If the set of Griffiths components of $Y$ did not depend on the choice of maximal semiorthogonal decomposition, 
then it would be a birational invariant~\cite[Lemma 3.10]{kuznetsov2015rationality}; 
in particular, it would be empty if $Y$ is rational of dimension at least $2$. 
Unfortunately, there are examples showing this is not true (see \cite[\S3.4]{kuznetsov2015rationality}, \cite{JH1}). 
It may be possible to salvage the situation by modifying the definition 
of a Griffiths component (some possibilities are discussed in \cite[\S3.4]{kuznetsov2015rationality}), 
but this remains an important question. 

Nonetheless, the existence of a Griffiths component appears to be related to 
irrationality in several examples. 
For instance, if $X' \subset \bP^5$ is a smooth cubic fourfold, there is a semiorthogonal decomposition 
\begin{equation}
\label{sod-cubic}
\Db({X'}) = \langle \cA_{X'}, \cO_{X'}, \cO_{X'}(1), \cO_{X'}(2) \rangle, 
\end{equation}
where $\cA_{X'}$ is a K3 category 
(see \cite[Corollary 4.3]{kuznetsov2004derived} or~\cite[Corollary~4.1]{kuznetsov2015calabi}). 
If $\cA_{X'}$ is equivalent to the derived category of a K3 surface, then $\gdim(\cA_{X'}) = 2$ 
and hence~\eqref{sod-cubic} contains no Griffiths components.
If $\cA_{X'}$ is not geometric (which holds for a very general cubic 
fourfold by an argument similar to Proposition~\ref{proposition-AX-geometric}), 
then we expect $\cA_{X'}$ to be a Griffiths component, although this remains 
an interesting open problem, cf. \cite[Conjecture~5.8]{kuznetsov2015calabi}.
These considerations motivated the following conjecture. 

\begin{conjecture}[\cite{kuznetsov2010derived}]
\label{conjecture-rationality-cubic}
If $X'$ is a rational cubic fourfold, then $\cA_{X'}$ is 
equivalent to the derived category of a K3 surface. 
\end{conjecture}

As evidence, this conjecture was proved in~\cite{kuznetsov2010derived} 
for all rational $X'$ known at the time. 
Since then, a nearly complete answer to when $\cA_{X'}$ is equivalent to the 
derived category of a K3 surface has been given~\cite{addington-thomas}, and some new 
families of rational cubic fourfolds have been produced~\cite{dp6-cubics}. 

The same philosophy can be applied to GM fourfolds. 
If the GM category $\cA_X$ of a GM fourfold $X$ is geometric, then~\eqref{eq:dbx-detailed} contains no Griffiths components,
and otherwise we expect~$\cA_X$ to be a Griffiths component.
This suggests the following analogue of Conjecture~\ref{conjecture-rationality-cubic}. 

\begin{conjecture}
\label{conjecture-rational-fourfold}
If $X$ is a rational GM fourfold, then the GM category $\cA_X$ is 
equivalent to the derived category of a $K3$ surface.
\end{conjecture}

One of the main results of this paper, Theorem~\ref{theorem-intro-K3} (or rather Theorem~\ref{theorem-associated-K3}), 
verifies Conjecture~\ref{conjecture-rational-fourfold} for a certain family of rational GM fourfolds. 
Another result, Theorem~\ref{theorem-intro-cubic} (or rather Theorem~\ref{theorem-associated-cubic}), builds a bridge between 
Conjectures~\ref{conjecture-rational-fourfold} and~\ref{conjecture-rationality-cubic}. 
Finally, recall that we proved the GM category of a very general GM fourfold 
is not equivalent to the derived category of a K3 surface 
(Proposition~\ref{proposition-AX-geometric}). 
Hence Conjecture~\ref{conjecture-rational-fourfold} is consistent with 
the expectation that a very general GM fourfold is irrational.

Now we consider GM varieties of other dimensions from the perspective of derived categories.
The next result shows that for a GM threefold $X$, any maximal semiorthogonal 
decomposition of $\Db(X)$ obtained by refining~\eqref{eq:dbx-detailed} contains a Griffiths component. 
We view this as evidence that \emph{any} smooth GM threefold is irrational.

\begin{lemma}[{cf.~\cite[Proposition~3.12]{kuznetsov2015rationality}}]
\label{lemma-threefold-AX}
Let $X$ be a GM threefold. Then $\cA_X$ does not admit a semiorthogonal decomposition 
with all components of geometric dimension at most $1$.  
\end{lemma}

\begin{proof}
It is easy to see that any category of geometric dimension $0$ is equivalent to 
$\Db(\Spec(\bC))$. 
Further, by~\cite{okawa} any category of geometric dimension $1$ 
is equivalent to the derived category of a curve. 
Note that $\HH_{\bullet}(\Spec(\bC)) = \bC$, and if $C$ is a curve of 
genus $g$ then 
\begin{equation*}
\HH_{\bullet}(C) = \bC^g[1] \oplus \bC^{2} \oplus \bC^g[-1]. 
\end{equation*} 
Thus if $\cA_X$ has a semiorthogonal decomposition with all components of geometric 
dimension at most $1$, Proposition~\ref{proposition-HH_*} and Theorem~\ref{theorem:additivity}  
imply $\cA_X \simeq \Db(C)$ for a genus $10$ curve $C$. 
This cannot happen by Proposition~\ref{proposition-AX-geometric}.
\end{proof}

If $X$ is a GM fivefold or sixfold, then by the discussion in \S\ref{subsection:duality}, 
$X$ has a generalized dual~$X^\vee$ with $\dim(X^\vee) \le \dim(X) - 2$. 
The duality conjecture (Conjecture~\ref{conjecture:duality}\eqref{conjecture-duality}) 
predicts that~$\cA_X \simeq \cA_{X^\vee}$, and hence $\gdim(\cA_X) \le \dim(X) - 2$.
So assuming the duality conjecture, we see that~\eqref{eq:dbx-detailed} has no Griffiths components, which 
is consistent with the rationality of $X$.



\section{Fourfold-to-surface duality}
\label{section-associated-K3}

In this section we prove Conjecture~\ref{conjecture:duality} 
for ordinary fourfolds with a generalized dual surface 
corresponding to a non-Pl\"{u}cker quadric point. 

\subsection{Statement of the result}
\label{subsection:statement-4-2-duality}

Recall that for any GM fourfold~$X$ and a quadric point $\bq \in \bP(V_6(X))$, 
we associated in \S\ref{subsection:duality} a generalized dual variety~$X_{\bq}^{\vee}$, 
which is an ordinary GM surface if~$\bq \in \sY_{\sA(X)}^3$.  

\begin{theorem}
\label{theorem-associated-K3}
Let $X$ be an ordinary GM fourfold such that 
\begin{equation*}
\sY_{\sA(X)}^{3} \cap (\bP(V_6(X)) \setminus \bP(V_5(X))) \neq \varnothing.   
\end{equation*} 
Then for any point $\bq \in \sY_{\sA(X)}^{3} \cap (\bP(V_6(X)) \setminus \bP(V_5(X)))$, 
there is an equivalence 
\begin{equation*}
\cA_X \simeq \Db(X_{\bq}^{\vee}). 
\end{equation*}
\end{theorem} 

The proof of this theorem takes the rest of this section.
We start by noting an immediate consequence for period partners. 

\begin{corollary}
Assume $X$ and $\bq$ are as in Theorem~\textup{\ref{theorem-associated-K3}}, 
and let $X_\bp$ be a period partner of~$X$ such that $(\bq, \bp)$ does not 
lie on the incidence divisor in $\bP(V_6(X)) \times \bP(V_6(X)^\vee)$. 
Then there is an equivalence of GM categories $\cA_{X_\bp} \simeq \cA_X$.
\end{corollary}

\begin{proof}
By Theorem~\ref{theorem-associated-K3} applied to $X$ and $X_\bp$ we have 
a pair of equivalences $\cA_X \simeq \Db(X_\bq^\vee)$ and $\cA_{X_\bp} \simeq \Db(X_\bq^\vee)$. 
Combining them we obtain an equivalence $\cA_{X_\bp} \simeq \cA_X$.
\end{proof}

A key ingredient in the proof of Theorem~\ref{theorem-associated-K3} is the 
theory of homological projective duality~\cite{kuznetsov2007HPD}. 
Very roughly, this theory relates the derived categories of linear sections 
of an ambient variety to those of orthogonal linear sections of a ``dual'' variety.   
As we explain below, the varieties $X$ and $X_\bq^{\vee}$ from Theorem~\ref{theorem-associated-K3}
can be thought of as intersections of $\bG \subset \bP(\wedge^2V_5)$ and 
its dual $\bGv = \G(2,V_5^\vee) \subset \bP(\wedge^2V_5^\vee)$ with projectively dual quadric subvarieties. 
To prove Theorem~\ref{theorem-intro-K3}, we thus establish a ``quadratic'' version of 
homological projective duality, in the case where the ambient variety is $\bG$. 
Much of our argument is not special to $\bG$, however, and should have interesting applications 
in other settings.

\begin{remark}
\label{remark-EWP3-divisor}
GM fourfolds $X$ as in the theorem form a \mbox{$23$-dimensional} (codimension $1$ in moduli) family. 
This can be seen using Theorem~\ref{theorem:dk:moduli}.
Indeed, by~\cite[Proposition~2.2]{ogrady2013double} Lagrangian subspaces $\sA \subset \wedge^3V_6$ 
with no decomposable vectors such that $\sY_\sA^3 \ne \varnothing$ form a divisor in the moduli space of all $\sA$, 
and hence form a $19$-dimensional family. 
Having fixed such an $\sA$ there are finitely many $\bq \in \sY_{\sA}^3$, 
and in order for $\bq \in \bP(V_6(X)) \setminus \bP(V_5(X))$ 
the Pl\"ucker point $\bp$ of $X$ can be any point of 
$\sY_{\sA^\perp}^1$ such that $(\bq,\bp)$ is not on the incidence divisor.
In other words, $\bp \in \sY_{\sA^\perp}^1 \setminus \bq^\perp$, so we have a 4-dimensional family of choices.
\end{remark}

Recall from \S\ref{subsection-classification} that if $X$ is an ordinary GM fourfold, 
there is a (canonical) hyperplane $W \subset \wedge^2V_5(X)$ and a 
(non-canonical) quadric $Q \subset \bP(W)$ such that $X = \bG \cap Q$. 
The fourfolds satisfying the assumption of Theorem~\ref{theorem-associated-K3} admit 
several different characterizations. 

\begin{lemma}
\label{lemma-GM-dP5}
Let $X$ be an ordinary GM fourfold. The following are equivalent:
\begin{enumerate}
\item $\sY^3_{\sA(X)} \cap ( \bP(V_6(X)) \setminus \bP(V_5(X)) ) \ne \varnothing$.
\item There is a rank $6$ quadric $Q \subset \bP(W)$ such that $X = \bG \cap Q$. 
\item $X$ contains a quintic del Pezzo surface, i.e. a smooth codimension $4$ linear 
section of the Grassmannian $\bG \subset \bP(\wedge^2V_5(X))$. 
\end{enumerate}
\end{lemma}

\begin{proof}
The equivalence of (1) and (2) follows from Proposition~\ref{proposition:dk:quadrics} since $\dim W = 9$.
Note that since $\sY_{\sA(X)}^4 = \varnothing$, the same proposition also shows that if $X = \bG \cap Q$ 
then $\rank(Q) \geq 6$. 

We show (2) is equivalent to (3). 
First assume (2) holds. 
Then a maximal isotropic space $I \subset W$ for $Q$ has 
dimension $6$, so $\bG \cap \bP(I)$ is a quintic del Pezzo contained in $X$, provided 
this intersection is transverse. By the argument of \cite[Lemma~4.1]{debarre-kuznetsov} 
(or by Lemma~\ref{lemma:pix:flat} below), this is true for a general $I$. 

Conversely, assume~(3) holds, i.e. assume there is a $6$-dimensional subspace 
$I \subset W$ such that $Z = \bG \cap \bP(I) \subset X$ is a quintic del Pezzo. 
The restriction map $V_6(X) \to \rH^0(\CMcal{I}_{Z/\bP(I)}(2))$ from quadrics 
in $\bP(W)$ containing $X$ 
to those in $\bP(I)$ containing $Z$ is surjective with one-dimensional kernel. 
If $Q \subset \bP(W)$ is the quadric corresponding to this kernel, 
then $X = \bG \cap Q$ and $\bP(I) \subset Q$. 
It follows that $\rank(Q) \leq 6$. But as we noted above, 
the reverse inequality also holds. 
\end{proof}

For the rest of the section, we fix an ordinary GM fourfold $X$ satisfying the equivalent conditions of Lemma~\ref{lemma-GM-dP5} 
and a point $\bq \in \sY_{\sA(X)}^{3} \cap (\bP(V_6(X)) \setminus \bP(V_5(X)))$. 
Further, to ease notation, we denote the generalized dual of $X$ corresponding to the quadric point $\bq$ (see Lemma~\ref{lemma:moduli-gm-duals-partners}) by 
\begin{equation*}
Y = X_{\bq}^{\vee}.
\end{equation*}
Note that $Y$ is a GM surface. 

\subsection{Setup and outline of the proof} 
\label{subsection-setup-associated-K3}
We outline here the strategy for proving Theorem~\ref{theorem-associated-K3}. 

The starting point 
is the following explicit geometric relation between $X$ and~$Y$. 
By Proposition~\ref{proposition:dk:quadrics},
the point $\bq$ corresponds to a rank $6$ quadric $Q$ cutting out~$X$, and 
the Pl\"ucker point~$\bp_X \in \bP(V_6(X)^\vee) \cong \bP(V_6(Y))$ of $X$ 
corresponds to a quadric $Q'$ cutting out~$Y$. 
Because $X$ and $Y$ are ordinary, we may regard $Q$ as a subvariety 
of $\bP(\wedge^2V_5(X))$ and $Q'$ as a subvariety of $\bP(\wedge^2V_5(Y))$.
Then~\cite[Proposition~3.28]{debarre-kuznetsov} (which is stated for 
dual varieties but works just as well for generalized duals) 
says that there is an isomorphism 
$V_5(X) \cong V_5(Y)^{\vee}$ identifying $Q' \subset \bP(\wedge^2V_5(Y))$ 
with the projective dual to \mbox{$Q \subset \bP(\wedge^2V_5(X))$}. 
Hence, fixing $V_5 = V_5(X)$, our setup is as follows: there is a hyperplane $W \subset \wedge^2V_5$ 
and a rank~$6$ quadric $Q \subset \bP(W)$ such that 
\begin{equation*}
X = \bG \cap Q
\qquad\text{and}\qquad
Y = \bG^\vee \cap Q^\vee,
\end{equation*}
where $Q^{\vee} \subset \bP(\wedge^2V_5^{\vee})$ is the projectively dual quadric to 
$Q \subset \bP(\wedge^2V_5)$, 
and 
\begin{equation*}
\bGv = \G(2,V_5^\vee) \subset \bP(\wedge^2V_5^{\vee})
\end{equation*}
is the dual Grassmannian. 

From this starting point, the main steps of the proof are as follows. 
First, by considering families of maximal linear subspaces of $Q$ and $\Qv$, 
we find $\bP^1$-bundles $\hX \to X$ and $\hY \to Y$, together with morphisms 
$\hX \to \bP^3$ and $\hY \to \bP^3$ realizing $\hX$ and $\hY$ as families of mutually 
orthogonal linear sections of $\bG$ and $\bGv$. 
This allows us to apply homological projective duality to obtain a semiorthogonal 
decomposition of $\Db(\hX)$ with $\Db(\hY)$ as a component. 
By comparing this (via mutation functors) with the decomposition of $\Db(\hX)$ coming from its~$\bP^1$-bundle structure over $X$, we show $\Db(\hY)$ has 
a decomposition into two copies of $\cA_X$. 
On the other hand, as $\hY \to Y$ is a \mbox{$\bP^1$-bundle}, 
$\Db(\hY)$ also decomposes into two copies of $\Db(Y)$. 
We show these two decompositions of $\Db(\hY)$ coincide, and 
hence $\cA_X \simeq \Db(Y)$. 
Our proof gives an explicit functor inducing this equivalence, 
see~\eqref{eq:equivalence:ax:y}.

\subsection{Maximal linear subspaces of the quadrics}
We start by discussing a geometric relation between $Q$ and $Q^\vee$.
Let $K \subset W$ be the kernel of $Q$, regarded as a symmetric linear map $W \to W^\vee$. 
Since $\dim W = 9$ and $\rank(Q)= 6$, we have $\dim K = 3$.
The filtration 
\begin{equation*}
0 \subset K \subset W \subset \wedge^2V_5
\end{equation*}
induces a filtration 
\begin{equation*}
0 \subset W^\perp \subset K^\perp \subset \wedge^2V_5^\vee
\end{equation*} 
where $K^\perp$ and $W^\perp$ are the annihilators of $K$ and $W$, 
so that $\dim K^\perp = 7$ and \mbox{$\dim W^\perp = 1$}.  
The pairing between the dual spaces $\wedge^2V_5$ and $\wedge^2V_5^\vee$ induces a nondegenerate pairing between~$W/K$ and $K^\perp/W^\perp$, and hence an isomorphism 
\begin{equation*}
 K^\perp/W^\perp \cong (W/K)^\vee.
\end{equation*}

The quadric $Q$ induces a smooth quadric $\oQ$ in the 
5-dimensional projective space $\bP(W/K)$.
The quadric $\oQ$ can be identified with the Grassmannian $\G(2,4)$; 
more precisely, we can find an isomorphism
\begin{equation*}
 W/K \cong \wedge^2 S
\end{equation*}
for a 4-dimensional vector space $S$, with an identification
\begin{equation*}
 \oQ = \G(2,S) \subset \bP(\wedge^2S). 
\end{equation*}
The projective dual of $\oQ$ is then the dual Grassmannian
\begin{equation*}
 \oQ^\vee = \G(2,S^\vee) \subset \bP(\wedge^2S^\vee) = \bP((W/K)^\vee) = \bP(K^\perp/W^\perp).
\end{equation*}
It follows that the projective dual of 
\begin{equation}\label{eq:q:cone}
 Q = \Cone_{\bP(K)}\oQ \subset \bP(\wedge^2V_5)
\end{equation}
is given by
\begin{equation}\label{eq:q:dual:cone}
 Q^\vee = \Cone_{\bP(W^\perp)}\oQ^\vee \subset \bP(\wedge^2V_5^\vee).
\end{equation}

Projective 3-space $\bP(S)$ is (a connected component of) the space of maximal 
linear subspaces of the quadric $\oQ = \G(2,S)$. 
The universal family is the flag variety $\Fl(1,2;S) \to \bP(S)$, with 
fiber over a point $s \in \bP(S)$ the plane $\bP(s\wedge S) \subset \bP(\wedge^2S)$.
Analogously, the same flag variety $\Fl(2,3;S^\vee) \cong \Fl(1,2;S)$ is (a connected component of) 
the space of maximal linear subspaces of $\oQ^\vee = \G(2,S^{\vee})$, 
this time with fiber over a point $s \in \bP(S)$
being the plane $\bP(\wedge^2s^\perp) \subset \bP(\wedge^2S^\vee)$. Note that the fibers of these two 
correspondences over a point $s \in \bP(S)$ are mutually orthogonal with respect to the pairing 
between $\wedge^2S$ and $\wedge^2S^\vee$.
We summarize this discussion by the diagram
\begin{equation}\label{eq:diagram:oq}
\vcenter{\xymatrix{
& \Fl(1,2;S) \ar[dl]_{p_{\oQ}} \ar[dr]^{\pi_{\oQ}} && \Fl(2,3;S^{\vee}) \ar[dl]_{\pi_{\oQ^\vee}} \ar[dr]^{p_{\oQ^\vee}} \\
\hbox to 0pt{\hss$\bP(\wedge^2S) \supset {}$} \oQ && \bP(S) && \oQ^\vee \hbox to 0pt{${} \subset \bP(\wedge^2S^\vee)$\hss}
}}
\end{equation}
with the inner arrows being $\bP^2$-bundles with mutually orthogonal fibers (as linear subspaces of~$\bP(\wedge^2S)$ and $\bP(\wedge^2S^\vee)$), 
and the outer arrows being $\bP^1$-bundles.

By~\eqref{eq:q:cone} every maximal isotropic subspace of $\oQ$ gives a maximal isotropic subspace of $Q$
by taking its preimage under the projection $W \to W/K = \wedge^2S$. Analogously, by~\eqref{eq:q:dual:cone} 
every maximal isotropic subspace of $\oQ^\vee$ gives a maximal isotropic subspace of $Q^\vee$ by taking 
its preimage under the projection $K^\perp \to K^\perp/W^\perp = \wedge^2S^\vee$. Note that for the pairing
between~$W$ and $K^\perp$ induced by the pairing between $\wedge^2V_5$ and $\wedge^2V_5^\vee$, 
the subspace $K \subset W$ is the left kernel, and the subspace $W^\perp \subset K^\perp$ is the right kernel. 
Hence any $s \in \bP(S)$ gives mutually orthogonal maximal isotropic spaces $\cI_s$ and $\cI_s^{\perp}$ 
of $Q$ and $Q^{\vee}$ respectively. 
These spaces form the fibers of vector bundles $\cI$ and $\cI^{\perp}$ over $\bP(S)$ of 
ranks $6$ and $4$, which are mutually orthogonal subbundles of $\wedge^2V_5 \otimes \cO_{\bP(S)}$ and 
$\wedge^2V_5^{\vee} \otimes \cO_{\bP(S)}$.
We can summarize this discussion by the following diagram
\begin{equation}
\label{eq:diagram:q}
\vcenter{\xymatrix{
& \bP_{\bP(S)}(\cI) \ar[dl]_{p_Q} \ar[dr]^{\pi_Q} && \bP_{\bP(S)}(\cI^\perp) \ar[dl]_{\pi_{Q^\vee}} \ar[dr]^{p_{Q^\vee}} 
\\
\hbox to 0pt{\hss{$\bP(\wedge^2V_5) \supset {}$}} Q && \bP(S) && Q^\vee \hbox to 0pt{{${} \subset \bP(\wedge^2V_5^\vee)$}\hss}
}}
\end{equation}
Here the inner arrows are $\bP^5$- and $\bP^3$-bundles with mutually orthogonal fibers, 
and the outer arrows are $\bP^1$-bundles (induced by the $\bP^1$-bundles
of diagram~\eqref{eq:diagram:oq}) away from the vertices $\bP(K)$ and $\bP(W^\perp)$ of the quadrics 
(over which the fibers are isomorphic to $\bP(S) \cong \bP^3$).

\subsection{Families of linear sections of the Grassmannians} 
Now define 
\begin{equation} 
\label{hX-hY}
\hX \coloneqq \bG \times_{\bP(\wedge^2V_5)} \bP_{\bP(S)}(\cI) \qquad \text{and} \qquad 
\hY \coloneqq \bGv \times_{\bP(\wedge^2V_5^{\vee})} \bP_{\bP(S)}(\cI^\perp)
\end{equation} 
to be the induced families of linear sections of $\bG$ and $\bGv$. 
They fit into a diagram 
\begin{equation}\label{eq:diagram:xy}
\vcenter{\xymatrix{
& \hX \ar[dl]_{p_X} \ar[dr]^{\pi_X} & & \hY \ar[dl]_{\pi_Y} \ar[dr]^{p_Y} \\
X && \bP(S) && Y
}}
\end{equation}
with the maps induced by those in~\eqref{eq:diagram:q}
(remember that $X = \bG \cap Q$ and $Y = \bG^\vee \cap Q^\vee$).

We will denote by $H,H'$, and $h$ the ample generators of 
$\Pic(\bG)$, $\Pic(\bGv)$, and $\bP(S)$.

\begin{lemma}
\label{lemma-hX-hY-P1}
There are rank $2$ vector bundles $\cS_X$ and $\cS_Y$ on $X$ and $Y$ with $\operatorname{c}_1(\cS_X) = -H$ and $\operatorname{c}_1(\cS_Y) = -H'$,
and isomorphisms 
\begin{equation*}
\hX \cong \bP_{X}(\cS_X) \qquad \text{and} \qquad 
\hY \cong \bP_{Y}(\cS_Y), 
\end{equation*}
such that $\cO_{\bP_{X}(\cS_X)}(1) = \pi_X^*\cO_{\bP(S)}(h)$ 
and $\cO_{\bP_{Y}(\cS_Y)}(1) = \pi_Y^*\cO_{\bP(S)}(h)$.
In particular, $\hX$ is a smooth fivefold, $\hY$ is a smooth threefold, and
\begin{equation}\label{eq:canclass:hx}
K_\hX = -H - 2h
\qquad\text{and}\qquad
K_\hY = H' -2h.
\end{equation}
\end{lemma}

\begin{proof}
Since $X$ and $Y$ are smooth, they do not intersect the vertices $\bP(K)$ and $\bP(W^\perp)$ of 
the quadrics $Q$ and $Q^{\vee}$,
hence the maps $p_X$ and $p_Y$ are $\bP^1$-fibrations induced by those in diagram~{\eqref{eq:diagram:q}}.
In other words, we have fiber product squares 
\begin{equation*}
\vcenter{
\xymatrix{
\hX \ar[r] \ar[d]_{p_X} & \Fl(1,2;S) \ar[d]^{p_{\oQ}} \\
X \ar[r] & \oQ
}
}
\qquad \text{and} \qquad
\vcenter{
\xymatrix{
\hY \ar[r] \ar[d]_{p_Y} & \Fl(2,3;S^{\vee}) \ar[d]^{p_{\oQ^\vee}} \\
Y \ar[r] & \oQ^\vee
}
}.
\end{equation*}
The map $p_{\oQ}$ is the projectivization of the tautological subbundle of $S \otimes \cO$
on $\oQ = \G(2,S)$,
and $p_{\oQ^\vee}$ is the projectivization of 
the annihilator of 
the tautological subbundle of $S^\vee \otimes \cO$ on~$\oQ^\vee = \G(2,S^\vee)$.
So we can take $\cS_X$ and $\cS_Y$ to be the pullbacks to $X$ and $Y$ of these bundles. 

To compute the canonical classes, note that the determinant of the tautological bundle 
(and of its annihilator) on $\G(2,S)$ is $\cO_{\G(2,S)}(-1)$, hence
$\operatorname{c}_1(\cS_X) = -H$ and $\operatorname{c}_1(\cS_Y) = -H'$. 
Now apply the standard formula for the canonical bundle of the projectivization of a vector bundle, 
taking into account that $K_X = -2H$ and $K_Y = 0$ by~\eqref{eq:omega-x}.
\end{proof}

\begin{lemma}\label{lemma:pix:flat}
The map $\pi_X\colon \hX \to \bP(S)$ is flat with general fiber a smooth quintic del Pezzo surface.
The map $\pi_Y\colon \hY \to \bP(S)$ is generically finite of degree $5$.
\end{lemma}

\begin{proof}
The fiber of $\pi_X$ over a point $s \in \bP(S)$ is the intersection $\bG \cap \bP(\cI_s)$, 
where the subspace $\bP(\cI_s) \subset \bP(\wedge^2V_5)$ has codimension $4$. 
Thus the dimension of $\pi_X^{-1}(s)$ is at least 2. 
If the dimension were greater than 2, this fiber would give a divisor in 
$X$ of degree at most $5$, but by~\eqref{eq:pic-x} and~\eqref{eq:omega-x}
every divisor in $X$ has degree divisible by $10$.
Thus every fiber is a dimensionally transverse intersection, and flatness follows.

Furthermore, since $\hX$ is smooth,
the general fiber of $\pi_X$ is a smooth quintic del Pezzo surface. 
Then by~\cite[Proposition~2.24]{debarre-kuznetsov}
the general fiber of $\pi_Y$ is a dimensionally transverse and 
smooth linear section of $\bG^\vee$ of codimension 6,
hence is just 5 reduced points.
\end{proof}

As a byproduct of the above, we obtain: 
\begin{lemma}
\label{lemma-rational-GM}
The variety $X$ is rational.
\end{lemma}

\begin{proof}
The same argument as in~\cite[Proposition~4.2]{debarre-kuznetsov} works. 
Let $\wtilde{X} \subset \hX$ be the preimage under the map~$\pi_X$ of a general hyperplane 
$\bP^2 \subset \bP(S)$. 
By Lemma~\ref{lemma:pix:flat}, the general fiber of~$\wtilde{X} \to \bP^2$ is a smooth quintic del Pezzo surface.
Hence by a theorem of Enriques~\cite{dP5-rational}, $\wtilde{X}$ is rational over~$\bP^2$, 
and so over~$\bC$ as well. On the other hand, the map $\wtilde{X} \to X$ is birational 
(in fact, it is a blowup of a quintic del Pezzo surface), so $X$ is rational too.
\end{proof}

\subsection{Homological projective duality}
\label{subsection:hpd}

Homological projective duality (HPD) is a key tool in the proof of 
Theorem~\ref{theorem-associated-K3}. 
Very roughly, HPD relates the derived categories of linear sections of a 
given variety to those of orthogonal linear sections of an ``HPD variety''. 
We refer to~\cite{kuznetsov2007HPD} for the details of this theory, 
and to~\cite{kuznetsov2014semiorthogonal} or~\cite{thomas2015notes} for an overview. 
For us, the relevant point is that the dual Grassmannian $\bGv$ 
is HPD to $\bG$. 
We spell out the precise consequence of this that we need below. 
 
Recall that by Lemma~\ref{lemma-sod-linear-section} 
there is a semiorthogonal decomposition 
\begin{equation*}
\label{ld-G}
\Db(\bG) = \langle \cB, \cB(H), \cB(2H), \cB(3H), \cB(4H) \rangle.  
\end{equation*}
Let
\begin{equation*}
i \colon \bH(\bG, \bG^{\vee}) \hookrightarrow \bG \times \bGv 
 \subset \bP(\wedge^2V_5) \times \bP(\wedge^2V_5^\vee) 
\end{equation*} 
be the incidence divisor defined by the canonical section of $\cO(H + H')$. 
Recall that $\cU$ denotes the tautological rank $2$ bundle on $\bG$, 
and let $\cV$ denote the tautological rank $2$ bundle on $\bG^{\vee}$. 
The following was proved in \cite[\S6.1]{kuznetsov2006hyperplane}. 
See~\cite[Definition 6.1]{kuznetsov2007HPD} for the definition of~HPD. 

\begin{theorem}
\label{theorem-G-HPD}
The Grassmannian $\bG^{\vee} \to \bP(\wedge^2V_5^{\vee})$ 
is HPD to $\bG \to \bP(\wedge^2V_5)$ with respect to the 
above semiorthogonal decomposition. 
Moreover, the duality 
is implemented by a sheaf~$\cE$ on~$\bH(\bG, \bGv)$ 
whose pushforward to $\bG \times \bGv$ fits into an exact sequence
\begin{equation*}
0 \to \cO_{\bG} \boxtimes \cV \to \cU^{\vee} \boxtimes \cO_{\bGv} 
\to i_* \cE \to 0. 
\end{equation*}
\end{theorem}

In fact, we shall only need a consequence of HPD, which we formulate below as 
Corollary~\ref{corollary-HPD}. 
Note that the natural map 
\begin{equation*}
\hX \times_{\bP(S)} \hY \to X \times Y \to \bG \times \bG^\vee
\end{equation*}
factors through $\bH(\bG, \bGv)$. Indeed, the fiber of 
$\hX \times_{\bP(S)} \hY$ over any point $s \in \bP(S)$
is 
\begin{equation*}
(\bP(\cI_s) \times \bP(\cI_s^\perp)) \cap (\bG \times \bG^\vee) \subset \bH(\bG,\bG^\vee). 
\end{equation*}
Note also that 
\begin{equation}\label{eq:dim:hxhy}
\dim(\hX \times_{\bP(S)} \hY) = 5, 
\end{equation} 
since the map $\hX \times_{\bP(S)} \hY \to \hY$ is flat of relative dimension 2 by Lemma~\ref{lemma:pix:flat}, 
and $\dim(\hY) = 3$ by Lemma~\ref{lemma-hX-hY-P1}.

Denote by $\hcE$ the pullback of 
the HPD object $\cE$ to $\hX \times_{\bP(S)} \hY$ and by $\hPhi\colon\Db(\hY) \to \Db(\hX)$ 
the corresponding Fourier--Mukai functor. 
Note that $\hPhi$ is $\bP(S)$-linear (since $\hcE$ is supported on the fiber product $\hX \times_{\bP(S)} \hY$), i.e. 
\begin{equation*}
\hPhi(\cF \otimes \pi_Y^*\cG) \cong \hPhi(\cF) \otimes \pi_X^*\cG
\end{equation*}
for all $\cF \in \Db(\hY)$ and $\cG \in \Db(\bP(S))$.
By Lemma~\ref{lemma-hX-hY-P1} and~\eqref{eq:dim:hxhy}, the families 
$\hX$ and $\hY$ of linear sections of $\bG$ and $\bGv$ 
satisfy the dimension assumptions of~\cite[Theorem 6.27]{kuznetsov2007HPD}. 
Hence we obtain: 

\begin{corollary}
\label{corollary-HPD}
The functor $\hPhi\colon\Db(\hY) \to \Db(\hX)$ is fully faithful, and there is a 
semiorthogonal decomposition
\begin{equation}
\label{eq:dbhx:f}
\Db(\hX) = \langle \hPhi(\Db(\hY)), \cB_X(H) \boxtimes \Db(\bP(S)) \rangle,
\end{equation}
where $\cB_X(H) \boxtimes \Db(\bP(S))$ denotes the triangulated subcategory 
generated by objects of the form $p_X^*(\cF) \otimes \pi_X^*(\cG)$
for $\cF \in \cB_X(H)$ and $\cG \in \Db(\bP(S))$.
\end{corollary}

\subsection{Mutations}
\label{section:duality-mutations}

Since $p_X\colon \hX \to X$ is a $\bP^1$-bundle (Lemma~\ref{lemma-hX-hY-P1}), 
we also have a semiorthogonal decomposition 
\begin{equation*}
\Db(\hX) = \langle p_X^* \Db(X), p_X^* \Db(X)(h) \rangle . 
\end{equation*}
Inserting the decomposition~\eqref{dbx} of $\Db(X)$, we obtain 
\begin{equation}
\label{eq:dbhx:s} 
\Db(X) = \langle \cA_{\hX}, \cB, \cB(H), \cA_{\hX}(h), \cB(h), \cB(H+h) \rangle, 
\end{equation}
where to ease notation we write $\cA_{\hX}$ for $p_X^*\cA_{X}$ and simply $\cB$ for $p_X^*\cB_X$. 
We find a sequence of mutations bringing this decomposition into the form of~\eqref{eq:dbhx:f}. 
In doing so we will use several times $K_X = -2H$, which holds by~\eqref{eq:omega-x}, and $K_\hX = -H - 2h$, which holds by~\eqref{eq:canclass:hx}. 
For a brief review of mutation functors and references, see the discussion in~\S\ref{subsection:self-duality}.

\medskip \noindent
\textbf{Step 1.} 
Mutate $\cB(H)$ to the left of $\cA_{\hX}$ in~\eqref{eq:dbhx:s}. Since this is a mutation in $p_X^*\Db(X)$ and $K_X = -2H$, 
by~\eqref{eq:mutation-serre-functor} we get
\begin{equation*}\label{eq:dbhx:1}
\Db(\hX) = \langle \cB(-H), \cA_{\hX}, \cB, \cA_{\hX}(h), \cB(h), \cB(H+h) \rangle.
\end{equation*}

\medskip \noindent
\textbf{Step 2.} 
Mutate $\cB(H+h)$ to the far left. Since $K_\hX = -H -2h$, 
by~\eqref{eq:mutation-serre-functor} we get
\begin{equation*}\label{eq:dbhx:2}
\Db(\hX) = \langle \cB(-h), \cB(-H), \cA_{\hX}, \cB, \cA_{\hX}(h), \cB(h) \rangle.
\end{equation*}

\medskip \noindent
\textbf{Step 3.} 
Mutate $\cB(-H)$ to the left of $\cB(-h)$. Since these two subcategories are completely orthogonal (see the lemma below), 
we get
\begin{equation*}\label{eq:dbhx:3}
\Db(\hX) = \langle \cB(-H), \cB(-h), \cA_{\hX}, \cB, \cA_{\hX}(h), \cB(h) \rangle.
\end{equation*}

\begin{lemma}
The categories $\cB(-H)$ and $\cB(-h)$ in $\Db(\hX)$ are completely orthogonal.
\end{lemma}
\begin{proof}
By Step 2, the pair $(\cB(-h), \cB(-H))$ is semiorthogonal. 
On the other hand, by Serre duality and~\eqref{eq:canclass:hx}, semiorthogonality of $(\cB(-H), \cB(-h))$ is equivalent 
to that of $(\cB(-h), \cB(2h))$, which follows from~\eqref{eq:dbhx:f} as $(\cO(-h), \cO(2h))$ 
is semiorthogonal in $\Db(\bP(S))$. 
\end{proof}

\medskip \noindent
\textbf{Step 4.} 
Mutate $\cB(-H)$ to the far right. Again by~\eqref{eq:mutation-serre-functor}, we get
\begin{equation*}\label{eq:dbhx:4}
\Db(\hX) = \langle \cB(-h), \cA_{\hX}, \cB, \cA_{\hX}(h), \cB(h), \cB(2h) \rangle.
\end{equation*}

\medskip \noindent
\textbf{Step 5.} 
Mutate $\cA_{\hX}$ and $\cA_{\hX}(h)$ to the far left. We get
\begin{align*}\label{eq:dbhx:5}
\nonumber \Db(\hX) & = \langle \rL_{\cB(-h)}(\cA_{\hX}), \rL_{\langle \cB(-h),\cB \rangle}(\cA_{\hX}(h)), \cB(-h), \cB, \cB(h), \cB(2h) \rangle \\ 
& = \langle \rL_{\cB(-h)}(\cA_{\hX}), \rL_{\langle \cB(-h),\cB \rangle}(\cA_{\hX}(h)), \cB_X \boxtimes \Db(\bP(S)) \rangle,  
\end{align*}
where we used the standard decomposition $\Db(\bP(S)) = \langle \cO(-h), \cO, \cO(h), \cO(2h) \rangle$. 

\medskip \noindent
\textbf{Step 6.} 
Twist the decomposition by $\cO(H)$. We get
\begin{equation}\label{eq:dbhx:6}
\Db(\hX) = \langle \rL_{\cB(H-h)}(\cA_{\hX}(H)), 
\rL_{\langle\cB(H-h),\cB(H)\rangle}(\cA_{\hX}(H+h)), \cB_X(H) \boxtimes \Db(\bP(S))  \rangle.
\end{equation}
To rewrite the first two components here, we used the following general fact: 
If $\cA \subset \cT$ is an admissible subcategory of a triangulated category 
and $F$ is an autoequivalence of $\cT$ (in our case $F$ is the  
autoequivalence of $\Db(\hX)$ given by tensoring with $\cO(H)$), 
then there is an isomorphism of functors 
\begin{equation}
\label{eq:mutation-equivalence}
F \circ \rL_\cA \cong \rL_{F(\cA)} \circ F. 
\end{equation}

Finally, we obtain: 

\begin{proposition}\label{proposition:hphis:ax:axh}
The functor $\hPhi^* \circ \Tnsr{\cO(H)} \colon\Db(\hX) \to \Db(\hY)$ induces an equivalence
\begin{equation*}
\langle \cA_{\hX}, \cA_{\hX}(h) \rangle \simeq \Db(\hY),
\end{equation*}
where $\hPhi^*\colon \Db(\hX) \to \Db(\hY)$ denotes the left adjoint of $\hPhi$. 
\end{proposition}

\begin{proof}
Comparing the decompositions~\eqref{eq:dbhx:6} and~\eqref{eq:dbhx:f}, we see that 
$\hPhi$ induces an equivalence 
\begin{equation*}
\hPhi\colon \Db(\hY) \xrightarrow{\ \sim\ } \langle \rL_{\cB(H-h)}(\cA_\hX(H)), \rL_{\langle \cB(H-h),\cB(H) \rangle}(\cA_\hX(H+h)) \rangle . 
\end{equation*}
Therefore its left adjoint $\hPhi^*$ gives an inverse equivalence. On the other hand, 
by semiorthogonality of~\eqref{eq:dbhx:f}
the functor $\hPhi^*$ vanishes on $\cB(H-h)$ and $\cB(H)$, hence its composition with the mutation functors 
through these categories is isomorphic to $\hPhi^*$. Thus $\hPhi^*$ induces an equivalence between 
$\langle \cA_\hX(H), \cA_\hX(H+h) \rangle \subset \Db(\hX)$
and~$\Db(\hY)$. This is equivalent to the claim.
\end{proof}

\subsection{Proof of the theorem}
Since $p_Y\colon \hY \to Y$ is a $\bP^1$-bundle (Lemma~\ref{lemma-hX-hY-P1}), we have 
\begin{equation}
\label{y:sod:p1}
\Db(\hY) = \langle p_Y^*\Db(Y), p_Y^*\Db(Y)(h) \rangle.
\end{equation}
We aim to prove that this semiorthogonal decomposition coincides with the one 
obtained by applying the fully faithful functor $ \Tnsr{\cO(-h)} \circ \hPhi^*\circ \Tnsr{\cO(H)}$ to $\langle \cA_{\hX}, \cA_{\hX}(h) \rangle$. 
For this, we consider the composition of functors
\begin{equation}
\label{definition:F}
\rF \coloneqq p_{Y*} \circ \Tnsr{\cO(-2h)} \circ \hPhi^* \circ \Tnsr{\cO(H)} \circ p_X^* \colon \Db(X) \to \Db(Y). 
\end{equation}

\begin{proposition}\label{proposition:rf:ax}
The functor $\rF$ vanishes on the subcategory $\cA_X \subset \Db(X)$.
\end{proposition}

Before proving the proposition, let us show how it implies the 
equivalence~$\cA_X \simeq \Db(Y)$.

\begin{proof}[Proof of Theorem~\textup{\ref{theorem-associated-K3}}]
We claim that 
\begin{equation}
\label{eq:equivalence:ax:y}
p_{Y*} \circ \Tnsr{\cO(-h)} \circ \hPhi^* \circ \Tnsr{\cO(H)} \circ p_X^* \colon \Db(X) \to \Db(Y)
\end{equation}
induces an equivalence $\cA_X \simeq \Db(Y)$. 
Note that the functor $p_X^*$ is fully faithful on $\cA_X$. 
So by Proposition~\ref{proposition:hphis:ax:axh} the functor 
\mbox{$\Tnsr{\cO(-h)} \circ \hPhi^* \circ \Tnsr{\cO(H)} \circ p_X^*$} gives a 
fully faithful embedding $\cA_X \hookrightarrow \Db(\hY)$, whose 
image $\cA$ satisfies 
\begin{equation}
\label{y:sod:a:ah}
\Db(Y) = \langle \cA, \cA(h) \rangle. 
\end{equation} 
On the other hand, by Proposition~\ref{proposition:rf:ax}  
the functor $p_{Y*}$ annihilates $\cA(-h)$. 
But the kernel of the functor $p_{Y*}$ is $p_Y^*\Db(Y)(-h)$, so 
$\cA(-h) \subset p_Y^*\Db(Y)(-h)$, and thus 
\begin{equation*}
\cA \subset p_Y^*\Db(Y) \qquad \text{and} \qquad 
\cA(h) \subset p_Y^*\Db(Y)(h). 
\end{equation*} 
In view of the decompositions~\eqref{y:sod:a:ah} and~\eqref{y:sod:p1}, 
we see that equality holds in the above inclusions. 
Since $p_{Y*}$ induces an equivalence $p_Y^*\Db(Y) \simeq \Db(Y)$, 
this finishes the proof. 
\end{proof}

Now we turn to the proof of Proposition~\ref{proposition:rf:ax}, which takes the rest of the section. 
Let $f_X\colon X \to \bG$ and $f_Y\colon Y \to \bGv$ be the Gushel maps,  
and let $p_{XY} \colon \hX \times_{\bP(S)} \hY \to X \times Y$ be the natural morphism.
Recall from \S\ref{subsection:hpd} that the composition 
\begin{equation*}
\hX \times_{\bP(S)} \hY \xrightarrow{\ p_{XY}\  }
X \times Y \xrightarrow{\, f_X \times f_Y\,} \bG \times \bG^{\vee}
\end{equation*}
factors through the incidence divisor $\bH(\bG, \bGv)$. 
Hence there is a commutative diagram 
\begin{equation}
\label{diagram:vanishing}
\vcenter{\xymatrix{
\hX \times_{\bP(S)} \hY \ar[r]^-{p} \ar[dr]_{\hg} & \bH(X,Y) \ar[d]^g \ar[r]^j & X \times Y \ar[d]^{f_X \times f_Y} \\
& \bH(\bG,\bGv) \ar[r]^{i} & \bG \times \bG^\vee
}}
\end{equation}
where $\bH(X,Y)$ is by definition the pullback of $\bH(\bG, \bGv)$ along $f_X \times f_Y$, 
and $p_{XY} = j \circ p$. 
We will need the following two lemmas.

\begin{lemma}\label{lemma:pf:o:hx:hy}
There is an isomorphism $p_* \cO_{\hX \times_{\bP(S)} \hY} \cong \cO_{\bH(X,Y)}$.
\end{lemma}

\begin{proof}
We have a diagram
\begin{equation*}
\xymatrix@C=6em{
\hX \times_{\bP(S)} \hY \ar[r]^-{\deh} \ar[d] & \hX \times \hY \ar[d]^{\pi_X \times \pi_Y} \ar[r]^{p_X \times p_Y} & X \times Y \\
\bP(S) \ar[r]^-{\Delta} & \bP(S) \times \bP(S)
}
\end{equation*}
where the square is cartesian, and also $\Tor$-independent 
as the fiber product has expected dimension by~\eqref{eq:dim:hxhy}.
To prove the lemma, we must show 
$(p_X \times p_Y)_* (\deh_*\cO_{\hX \times_{\bP(S)} \hY}) \cong \cO_{\bH(X,Y)}$. 
By $\Tor$-independence, we have an isomorphism 
\begin{equation*}
\deh_*\cO_{\hX \times_{\bP(S)} \hY} \cong (\pi_X \times \pi_Y)^* \Delta_* \cO_{\bP(S)}. 
\end{equation*}
Pulling back the standard resolution of the diagonal on $\bP(S) \times \bP(S)$, we obtain 
an exact sequence
\begin{multline*}
0 \to 
\pi_X^*\cO_{\bP(S)}(-3h) \boxtimes \pi_Y^*\Omega_{\bP(S)}^3(3h) \to
\pi_X^*\cO_{\bP(S)}(-2h) \boxtimes \pi_Y^*\Omega_{\bP(S)}^2(2h) \to 
\\ \to
\pi_X^*\cO_{\bP(S)}(-h) \boxtimes \pi_Y^*\Omega_{\bP(S)}^1(h) \to
\cO_{\hX \times \hY} \to 
\deh_*\cO_{\hX \times_{\bP(S)} \hY} \to 0
\end{multline*}
on $\hX \times \hY$. 
Using the identifications $p_X\colon \hX = \bP_X(\cS_X) \to X$ and $p_Y\colon\hY = \bP_Y(\cS_Y) \to Y$ 
of Lemma~\ref{lemma-hX-hY-P1}, it is easy to compute: 
\begin{align*}
p_{Y*}\pi_Y^*\Omega_{\bP(S)}^3(3h) & \cong p_{Y*}\pi_Y^*\cO(-h) = 0 , \\ 
p_{X*}\pi_X^*\cO_{\bP(S)}(-2h) & \cong \det(\cS_X)[-1] \cong \cO_X(-H)[-1] , \\
p_{Y*} \pi_Y^*\Omega_{\bP(S)}^2(2h) & \cong 
\det(\cS_Y) \cong \cO_Y(-H'), \\
p_{X*}\pi_X^*\cO_{\bP(S)}(-h) & = 0, \\ 
(p_X \times p_Y)_*(\cO_{\hX \times \hY}) & \cong \cO_{X \times Y}.
\end{align*}
It follows that in the spectral sequence computing 
$(p_X \times p_Y)_*(\deh_*\cO_{\hX \times_{\bP(S)} \hY})$ from the above resolution, 
the only nontrivial terms are 
\begin{align*}
\rR^1(p_X \times p_Y)_*(\pi_X^*\cO(-2h) \boxtimes \pi_Y^*\Omega_{\bP(S)}^2(2h)) & 
\cong \cO_{X \times Y}(-H-H'), \\
\rR^0 (p_X \times p_Y)_*(\cO_{\hX \times \hY}) & 
\cong \cO_{X \times Y},
\end{align*}
and we get an exact sequence
\begin{equation*}
0 \to \cO_{X\times Y}(-H-H') \to \cO_{X \times Y} 
\to (p_X \times p_Y)_* (\deh_*\cO_{\hX \times_{\bP(S)} \hY}) \to 0,
\end{equation*}
which gives the required isomorphism 
$(p_X \times p_Y)_* (\deh_*\cO_{\hX \times_{\bP(S)} \hY}) \cong \cO_{\bH(X,Y)}$. 
\end{proof}

\begin{lemma}\label{lemma:F:FM}
The functor $\rF [-2]$ is given by a Fourier--Mukai
kernel $\cK \in \Db(X \times Y)$, which fits into a distinguished triangle 
\begin{equation*}
\cU_X(-H) \boxtimes \cO_Y(-H') \to \cO_X(-H) \boxtimes \cV_Y^\vee(-H') \to \cK.
\end{equation*}
\end{lemma}

\begin{proof}
The main term in the definition~\eqref{definition:F} of $\rF$ is the left adjoint $\hPhi^*$ 
of $\hPhi$. 
By definition $\hPhi$ is given by the Fourier--Mukai kernel $\hcE \in \Db(\hX \times_{\bP(S)} \hY)$, 
so by Grothendieck duality we find that $\hPhi^*$ is given by the kernel 
\begin{equation*}
\hcE^{\vee} \otimes \omega_{\hX \times_{\bP(S)} \hY/\hY}[2] 
= \hcE^{\vee}(2h-H)[2] \in \Db(\hX \times_{\bP(S)} \hY),
\end{equation*}
where $\hcE^{\vee} = \RCHom(\hcE,\cO)$ is the derived dual of $\hcE$ on $\hX \times_{\bP(S)} \hY$. 
Using this, it follows easily from the definition of $\rF$ that $\rF[-2]$ is given by 
the kernel
\begin{equation*}
\cK \coloneqq p_{XY*}(\hcE^{\vee}) \in \Db(X \times Y). 
\end{equation*}
Using the diagram~\eqref{diagram:vanishing} and the definition of $\hcE$, 
we can rewrite this as
\begin{align*}
\cK & \cong j_*p_* \RCHom(p^*g^*\cE, \cO_{\hX\times_{\bP(S)}\hY}) \\
& \cong j_*\RCHom(g^*\cE, p_*\cO_{\hX\times_{\bP(S)}\hY}) \\ 
& \cong j_*\RCHom(g^*\cE, \cO_{\bH(X,Y)}), 
\end{align*}
where the second line holds by the local adjunction between $p^*$ and $p_*$,  
and the third by Lemma~\ref{lemma:pf:o:hx:hy}. 
Now Grothendieck duality for the inclusion $j\colon \bH(X,Y) \to X \times Y$ of the incidence 
divisor (which has class $H + H'$) gives
\begin{equation*}
\cK \cong 
j_*\RCHom(g^*\cE, j^!\cO_{X \times Y}(-H-H')[1]) \cong
\RCHom(j_*g^*\cE, \cO_{X \times Y}(-H-H')[1]). 
\end{equation*}

On the other hand, the fiber square in diagram~\eqref{diagram:vanishing} is 
$\Tor$-independent because $\bH(X,Y)$ has expected dimension.  
Hence we have an isomorphism 
\begin{equation*}
j_*g^*\cE \cong (f_X \times f_Y)^* i_*\cE,
\end{equation*}
and so, by the explicit resolution of $i_*\cE$ from Theorem~\ref{theorem-G-HPD}, a distinguished triangle 
\begin{equation*}
\cO_X \boxtimes \cV_Y \to \cU_X^{\vee} \boxtimes \cO_Y \to j_*g^*\cE. 
\end{equation*}
Dualizing, twisting by $\cO_{X \times Y}(-H-H')$, and rotating this triangle, 
we obtain a distinguished triangle
\begin{equation*}
\cU_X(-H) \boxtimes \cO_Y(-H') \to \cO_X(-H) \boxtimes \cV_Y^\vee(-H') \to 
\RCHom(j_*g^*\cE, \cO_{X \times Y}(-H-H')[1]),
\end{equation*}
which combined with the above expression for $\cK$ finishes the proof.
\end{proof}

Finally, we prove Proposition~\ref{proposition:rf:ax}. 
\begin{proof}[Proof of Proposition~\textup{\ref{proposition:rf:ax}}]
By Lemma~\ref{lemma:F:FM}, it suffices to show the Fourier--Mukai functors with kernels 
\begin{equation*}
\cU_X(-H) \boxtimes \cO_Y(-H') \qquad \text{and} \qquad 
\cO_X(-H) \boxtimes \cV_Y^\vee(-H') 
\end{equation*}
vanish on $\cA_X$. This is equivalent to the vanishing 
\begin{equation*}
\rH^{\bullet}(X, \cU_X(-H) \otimes \cF) = 0
\quad \text{and} \quad 
\rH^{\bullet}(X, \cO_X(-H) \otimes \cF) = 0
\end{equation*}
for all $\cF \in \cA_X$, which holds since $\cA_X$ is 
right orthogonal to $\cB_X(H) = \langle \cO_X(H), \cU^\vee_X(H) \rangle$ 
by definition (see~\eqref{dbx} and~\eqref{equation-definition-B}). 
\end{proof}


\section{Cubic fourfold derived partners}
\label{section-associated-cubic}

In this section, we show that the K3 categories attached to GM and cubic fourfolds not only behave similarly, 
but sometimes even coincide. For this, we will consider ordinary GM fourfolds satisfying the following 
condition: there is a $3$-dimensional subspace $V_3 \subset V_5(X)$ such that
\begin{equation}
\label{eq:GM:plane}
\G(2,V_3) \subset X.
\end{equation}

\begin{remark}
\label{remak:GM4-plane}
GM fourfolds that satisfy~\eqref{eq:GM:plane} for some $V_3$ form a \mbox{$21$-dimensional} (codimension~$3$ in moduli) family. 
This can be seen using Theorem~\ref{theorem:dk:moduli}, as follows. 
Let $V_6 = V_6(X)$. 
Then by~\cite[Theorem~4.5(c)]{debarre-kuznetsov-periods}, for a $3$-dimensional 
subspace $V_3 \subset V_6$ condition~\eqref{eq:GM:plane} holds if and only if
\begin{equation}
\label{eq:GM-plane-condition}
\dim (\sA \cap ((\wedge^2V_3) \wedge V_6)) \geq 4
\qquad\text{and}\qquad 
\bp_X \in \bP(V_3^\perp) \subset \bP(V_6^\vee). 
\end{equation}
By~\cite[Lemma~3.6]{ikkr} Lagrangian subspaces $\sA \subset \wedge^3V_6$ with no decomposable vectors 
such that the first part of~\eqref{eq:GM-plane-condition} holds for some $V_3 \subset V_6$ 
form a nonempty divisor in the moduli space of all $\sA$, and hence form a $19$-dimensional family. 
Having fixed such an $\sA$ there are finitely many points $V_3 \in \G(3,V_6)$ such that 
the first part of~\eqref{eq:GM-plane-condition} holds \cite{email}. 
By Theorem~\ref{theorem:dk:moduli}, for such a~$V_3$, the ordinary GM fourfolds $X$ such that the second part of~\eqref{eq:GM-plane-condition} 
holds are parameterized by~$\sY_{\sA^\perp}^1 \cap \bP(V_3^\perp)$.  
By~\cite[Lemma~2.3]{debarre-kuznetsov-periods} we have $\bP(V_3^\perp) \subset \sY_{\sA^\perp}$. 
Further, since~$\sY_{\sA^\perp}^{\ge 2}$ is an irreducible surface of degree 40, we have $\bP(V_3^\perp) \not\subset \sY_{\sA^\perp}^{\ge 2}$. 
Thus $\sY_{\sA^\perp}^1 \cap \bP(V_3^\perp)$ is an open subset of the projective plane $\bP(V_3^\perp)$.
\end{remark}

From now on we write $V_5 = V_5(X)$ and fix a $3$-dimensional subspace $V_3 \subset V_5$ such that~\eqref{eq:GM:plane} holds. 
We associate to $X$ a birational cubic fourfold~$X'$ 
following~\cite[\S7.2]{DIM4fold}.  
Generically $X'$ is smooth, and in this case we prove there is an equivalence 
$\cA_X \simeq \cA_{X'}$ where $\cA_{X'}$ is the K3 category of the cubic fourfold 
defined by \eqref{sod-cubic} (Theorem~\ref{theorem-associated-cubic}).
The cubic $X'$ is simply the image of the linear projection 
from the plane $\G(2,V_3)$ in~$X$. 
We begin by studying this projection as a map from the entire Grassmannian~$\bG$.

\subsection{A linear projection of the Grassmannian} 
Set
\begin{equation*} 
P = \bP(\wedge^2V_3) = \G(2,V_3) \subset \bG.
\end{equation*} 
Choose a complement $V_2$ to $V_3$ in $V_5$, and set 
\begin{equation*}
B = \wedge^2V_5 / {\wedge^2V_3} = 
\wedge^2 V_2 \oplus (V_2 \otimes V_3).
\end{equation*}
Then the linear projection from $P$ gives a birational isomorphism from $\bG$ to $\bP(B)$. 
Its structure can be described as follows.

\begin{lemma}
\label{lemma-gr-projection}
Let $p\colon\wtilde\bG \to \bG$ be the blowup with center in $P$. 
Then the linear projection from $P$ induces a regular map $q\colon\wtilde\bG \to \bP(B)$ which identifies $\bG$ 
with the blowup of $\bP(B)$ in~$\bP(V_2) \times \bP(V_3) \subset \bP(V_2 \otimes V_3) \subset \bP(B)$.
In other words, we have a diagram
\begin{equation}
\label{grdiag}
\vcenter{\xymatrix{
&
E \ar[r] \ar[dl] &
\wtilde{\bG} \ar[dl]_{p} \ar[dr]^{q}  &
E' \ar[l] \ar[dr]
\\
P \ar[r] &
\bG &&
\bP(B) &
\bP(V_2) \times \bP(V_3) \ar[l]
}}
\end{equation}
where
\begin{itemize}
\item 
$E$ is the exceptional divisor of the blowup $p$, and is mapped birationally by $q$ 
onto the hyperplane $\bP(V_2 \otimes V_3) \subset \bP(B)$. 
\item 
$E'$ is the exceptional divisor of the blowup $q$, and is mapped birationally by $p$ onto 
the Schubert variety
\begin{equation*}
\Sigma = \left \{ U \in \bG ~ | ~ U \cap V_3 \neq 0 \right \} \subset \bG
\end{equation*}
\end{itemize}
\end{lemma}
\begin{proof}
Straightforward. 
\end{proof}

We denote by $H$ and $H'$ the ample generators of $\Pic(\bG)$ and 
$\Pic(\bP(B))$. 

\begin{lemma}
On $\wtilde{\bG}$ we have the relations 
\begin{equation}
\label{divisor-relations}
\left\{ \begin{aligned}
H' & = H - E, \\
E' & = H - 2E, 
\end{aligned} \right. 
\qquad \text{or equivalently} \qquad 
\left\{ \begin{aligned}
H & = 2H'-E', \\
E & = H'-E',
\end{aligned} \right.
\end{equation}
as divisors modulo linear equivalence. Moreover, we have
\begin{equation}
\label{eq:canclass:tbg}
K_{\wtilde\bG} = -5H + 3E = -7H' + 2E'.
\end{equation} 
\end{lemma}

\begin{proof}
The equalities~\eqref{eq:canclass:tbg} follow from the standard formula 
for the canonical class of a blowup, and 
the equality $H' = H - E$ holds by definition of $p$. 
Using these, the other equalities in~\eqref{divisor-relations} follow 
directly (note that $\Pic(\wtilde{\bG}) \cong \bZ^2$ is torsion free). 
\end{proof}

Later in this section we will need an expression for the vector bundle 
$p^*\cU^{\vee}$ on $\wtilde{\bG}$ in terms of the blowup $q$.
For this, we consider the composition 
\begin{multline*}
\phi\colon (V_2^\vee \oplus V_3^\vee) \otimes \cO_{\bP(B)} 
\hookrightarrow
V_2 \otimes V_2^\vee \otimes (V_2^\vee \oplus V_3^\vee) \otimes \cO_{\bP(B)} \to \\ \to
V_2 \otimes (\wedge^2 V_2^\vee  \oplus (V_2^\vee \otimes V_3^\vee)) \otimes \cO_{\bP(B)} 
\to V_2 \otimes \cO_{\bP(B)}(H'),
\end{multline*}
where the first morphism is induced by the map $\bC \to V_2 \otimes V_2^{\vee}$ corresponding 
to the identity of $V_2$, 
the second is induced by the map $V_2^{\vee} \otimes V_2^{\vee} \to \wedge^2 V_2^{\vee}$, 
and the third is induced by the composition
\begin{equation*}
(\wedge^2V_2^\vee \oplus (V_2^\vee \otimes V_3^\vee)) \otimes \cO_{\bP(B)} 
= B^\vee \otimes \cO_{\bP(B)} \to 
\cO_{\bP(B)}(H').
\end{equation*}

\begin{lemma}
\label{lemma-cokernel-phi}
The cokernel of $\phi$ is the sheaf $\cO_{\bP(V_2) \times \bP(V_3)}(2,1)$.
\end{lemma}

\begin{proof}
Write 
\begin{align*}
\phi' & \colon  V_2^\vee \otimes \cO_{\bP(B)} \to V_2 \otimes \cO_{\bP(B)}(H'), \\
\phi'' & \colon V_3^\vee \otimes \cO_{\bP(B)} \to V_2 \otimes \cO_{\bP(B)}(H'),
\end{align*}
for the components of $\phi$. 
The morphism $\phi'$ is an isomorphism away from the hyperplane $\bP(V_2 \otimes V_3) \subset \bP(B)$, 
and zero on it. 
Hence $\coker(\phi') = V_2 \otimes \cO_{\bP(V_2 \otimes V_3)}(H')$. 
It follows that the cokernel of $\phi$ coincides with the cokernel of the morphism   
\begin{equation*}
\phi''_{|\bP(V_2 \otimes V_3)}\colon V_3^\vee \otimes \cO_{\bP(V_2 \otimes V_3)} \to V_2 \otimes \cO_{\bP(V_2 \otimes V_3)}(H').
\end{equation*}
But the morphism $\phi''_{|\bP(V_2 \otimes V_3)}$ is generically surjective with degeneracy 
locus the Segre subvariety $\bP(V_2) \times \bP(V_3) \subset \bP(V_2 \otimes V_3)$, and 
its restriction to this locus factors as the composition 
\begin{equation*}
V_3^\vee  \otimes \cO_{\bP(V_2)\times \bP(V_3)} \twoheadrightarrow \cO_{\bP(V_2)\times \bP(V_3)}(0,1) \hookrightarrow V_2 \otimes \cO_{\bP(V_2)\times \bP(V_3)}(1,1) = V_2 \otimes \cO_{\bP(V_2)\times \bP(V_3)}(H').
\end{equation*}
It follows that the cokernel of $\phi''_{|\bP(V_2 \otimes V_3)}$ is isomorphic to $\cO_{\bP(V_2)\times \bP(V_3)}(2,1)$. 
\end{proof}

Let $F$ denote the class of a fiber of the natural projection $E' \to \bP(V_2) \times \bP(V_3) \to \bP(V_2)$.
\begin{proposition}
On $\wtilde{\bG}$ there is an exact sequence
\begin{equation}
\label{us}
0 \to p^*\cU^\vee \to V_2 \otimes \cO_{\wtilde{\bG}}(H') \to \cO_{E'}(H' + F) \to 0.
\end{equation} 
\end{proposition}

\begin{proof}
By Lemma~\ref{lemma-cokernel-phi}, we have an exact sequence
\begin{equation*}
V_5^\vee \otimes \cO_{\bP(B)} \xrightarrow{\, \phi \,} V_2 \otimes \cO_{\bP(B)}(H') \to 
\cO_{\bP(V_2) \times \bP(V_3)}(2,1) \to 0.
\end{equation*}
Pulling back to $\wtilde{\bG}$, we obtain an exact sequence
\begin{equation*}
V_5^\vee \otimes \cO_{{\wtilde{\bG}}} \to V_2 \otimes \cO_{{\wtilde{\bG}}}(H') \to \cO_{E'}(H' + F) \to 0,
\end{equation*} 
Since $E'$ is a divisor on $\wtilde\bG$, the kernel $\cK$ of the epimorphism $V_2 \otimes \cO_{{\wtilde{\bG}}}(H') \to \cO_{E'}(H' + F)$ is 
a rank $2$ vector bundle on~$\wtilde{\bG}$, which by the above exact sequence 
is a quotient of the trivial bundle $V_5^\vee \otimes \cO_{\wtilde{\bG}}$. 
Hence $\cK$ induces a morphism $\wtilde{\bG} \to \bG$. 
This morphism can be checked to agree with the blowdown morphism~$p$, so 
$\cK \cong p^*\cU^\vee$. 
\end{proof}

\subsection{Setup and statement of the result}
Recall that $X$ is an ordinary GM fourfold containing the plane $P = \G(2,V_3)$. 
The following proposition describes the structure of the rational map from 
$X$ to $\bP^5$ given by projection from $P$. 
We slightly abuse notation by using the same symbols for 
the exceptional divisors and blowup morphisms as in the 
above discussion of~$\bG$.

\begin{proposition}
\label{proposition-blowup-X}
Let $p \colon \wtilde{X} \to X$ be the blowup with center in $P$. 
Then the linear projection from $P$ induces a regular map 
$q\colon \wtilde{X} \to X'$ to a cubic fourfold $X'$ containing a smooth cubic surface scroll\/~$T$, 
and identifies $\wtilde{X}$ as the blowup of $X'$ in $T$. 
In other words, we have a diagram 
\begin{equation*}
\vcenter{\xymatrix{
&
E \ar[r]^{i} \ar[dl]_{p_E} &
\wtilde{X} \ar[dl]_{p} \ar[dr]^{q} &
E' \ar[l]_{j} \ar[dr]^{q_{E'}} 
\\
P \ar[r] &
X &&
X' &
T \ar[l]
}}
\end{equation*}
where $p$ and $q$ are blowups with exceptional 
divisors $E$ and $E'$. 
Moreover, the relations~\eqref{divisor-relations} continue to hold on 
$\wtilde{X}$, and 
\begin{equation}\label{eq:canclass:tx}
K_{\wtilde{X}} = -2H + E = -3H' + E'.
\end{equation}
Finally, if $X$ does not contain planes of the form $\bP(V_1 \wedge V_4)$ where $V_1 \subset V_3 \subset V_4 \subset V_5$, then~$X'$ is smooth.
\end{proposition}

\begin{proof}
By \S\ref{subsection-classification}  
there is a hyperplane $\bP(W) \subset \bP(\wedge^2V_5)$ and 
a quadric hypersurface $Q \subset \bP(W)$ such that $X = \bG \cap Q$ 
and $P \subset Q$. 
Consider the subspace
\begin{equation*}
C = W/{\wedge^2V_3} \subset \wedge^2V_5/{\wedge^2V_3} = B,
\end{equation*}
so that $\bP(C) \subset \bP(B)$ is a hyperplane. 
We claim that the corresponding hyperplane section 
\begin{equation*}
T = (\bP(V_2) \times \bP(V_3)) \cap \bP(C)
\end{equation*}
of $\bP(V_2) \times \bP(V_3) \subset \bP(B)$ is a smooth cubic surface scroll. 
For this it is enough to show that~$\bP(C) \cap \bP(V_2 \otimes V_3)$ is a 
hyperplane in $\bP(V_2 \otimes V_3)$ whose equation, considered as an element 
in $V_2^\vee \otimes V_3^\vee \cong \Hom(V_3,V_2^\vee)$, has rank 2.
Assume on the contrary that the rank of this equation is at most 1.
Then its kernel is a subspace of $V_3$ of dimension at least~2, which is contained 
in the kernel of the skew form $\omega$ on $V_5$ defining $W$. 
So the rank of $\omega$ is $2$.
But then the Grassmannian hull $M_X = \bG \cap \bP(W)$ of $X$ is singular 
along $\bP^2 = \G(2, \ker(\omega))$, and~$X$ is singular along $\bP^2 \cap Q$.
This contradiction proves the claim.

The proper transform of the Grassmannian hull $M = M_X$ under 
$p \colon \wtilde\bG \to \bG$ coincides with the proper transform 
of $\bP(C)$ under $q \colon \wtilde\bG \to \bP(B)$. 
Thus if $\wtilde{M} = \Bl_P(M) \to M$ is the blowup in $P$, 
then projection from $P$ gives an identification 
$\wtilde{M} \cong \Bl_T(\bP(C)) \to \bP(C)$. 
Further, the proper transform of 
\mbox{$X = M \cap Q$} under $\wtilde{M} \to M$ is cut out by a section of 
the line bundle
\begin{equation*}
\cO_{\wtilde{M}}(2H - E) = \cO_{\wtilde{M}}(3H'-E'), 
\end{equation*}
and therefore coincides with the proper transform under $\wtilde{M} \to \bP(C)$ 
of a cubic fourfold \mbox{$X' \subset \bP(C)$} containing~$T$.  
This proves the first part of the lemma. 

The relations~\eqref{divisor-relations} clearly restrict to $\wtilde{X}$,
and the equalities~\eqref{eq:canclass:tx} follow from the standard 
formula for the canonical class of a blowup. 

It remains to show that $X'$ is smooth if $X$ does not contain planes of the form $\bP(V_1 \wedge V_4)$ where $V_1 \subset V_3 \subset V_4 \subset V_5$. 
For this, first note that the blowup of $X'$ in $T$ is smooth, 
since it coincides with the blowup of $X$ in $P$.
Therefore, $X'$ is smooth away from $T$. On the other hand, $T$ is also smooth, 
so it is enough to check that $T \subset X'$ is a locally complete intersection,
i.e.\ that its conormal sheaf is locally free. Since $E' \to T$ is the exceptional divisor 
of the blowup of $X'$ in $T$, it is enough to check that the map $E' \to T$ is a $\bP^1$-bundle.
Since $E'$ is cut out in the exceptional divisor of~\eqref{grdiag} by fiberwise linear conditions, 
it is enough to show that there are no points in $T \subset \bP(V_2) \times \bP(V_3)$
over which the fiber of $E'$ is isomorphic to $\bP^2$. 
But such a point would correspond to a choice of a $V_1 \subset V_3$ (giving a point in $\bP(V_3)$) 
and $V_3 \subset V_4$ (giving a point of $\bP(V_5/V_3) = \bP(V_2)$), such that the plane 
$\bP(V_1 \wedge V_4)$ is in~$X$. Since we assumed there are no such planes in~$X$, we 
conclude that $X'$ is smooth. 
\end{proof}

The condition guaranteeing smoothness of $X'$ in the final statement of 
Proposition~\ref{proposition-blowup-X} holds generically: 

\begin{lemma}
If $X$ is a general ordinary GM fourfold containing $P = \G(2,V_3)$ 
for some $V_3 \subset V_5$,  
then $X$ does not contain planes of the form $\bP(V_1 \wedge V_4)$ 
where $V_1 \subset V_3 \subset V_4 \subset V_5$.
\end{lemma}

\begin{proof}
By Theorem~\ref{theorem:dk:moduli}, an ordinary GM fourfold $X$ corresponds to a pair $(\sA, \bp)$ 
such that~$\sA$ has no decomposable vectors and $\bp \in \sY_{\sA^\perp}^1$.
By Remark~\ref{remak:GM4-plane}, $X$ contains the plane $\G(2,V_3)$ if and only if~\eqref{eq:GM-plane-condition} holds.
Similarly, by~\cite[Theorem~4.3(c)]{debarre-kuznetsov-periods}, $X$ contains a plane 
$\bP(V_1 \wedge V_4)$ if and only if $\sY^3_\sA \cap \bP(V_5) \ne \emptyset$.

By~\cite[Lemma~3.6]{ikkr} Lagrangians $\sA \subset \wedge^3V_6$ with no decomposable vectors 
such that there is $V_3 \subset V_6$ for which the first part of~\eqref{eq:GM-plane-condition} holds
are parameterized by an open subset of a divisor $\Gamma \subset \LGr(10,\wedge^3V_6)$, 
and by~\cite[Lemma~3.7]{ikkr} this divisor has no common components with the divisor $\Delta \subset \LGr(10,\wedge^3V_6)$ 
parameterizing~$\sA$ such that $\sY_\sA^3 \ne \emptyset$.
Choose any $\sA$ with no decomposable vectors 
such that there is $V_3 \subset V_6$ for which the first part of~\eqref{eq:GM-plane-condition} holds, 
but $\sY_\sA^3 = \emptyset$.
Then as explained in Remark~\ref{remak:GM4-plane}, there is a 2-dimensional family of ordinary GM fourfolds 
containing $\G(2,V_3)$; none of these contain a plane of the form $\bP(V_1 \wedge V_4)$ since~$\sY_\sA^3 = \emptyset$.
\end{proof}

Our goal is to prove the following result. 

\begin{theorem}
\label{theorem-associated-cubic}
Assume the cubic fourfold $X'$ associated to $X$ by Proposition~\textup{\ref{proposition-blowup-X}} is smooth.
Then there is an equivalence $\cA_{X} \simeq \cA_{X'}$, 
where $\cA_X$ is the GM category defined by~\eqref{eq:dbx-detailed} and~$\cA_{X'}$ is defined by~\eqref{sod-cubic}.
\end{theorem}

\begin{remark}
\label{remark-not-K3-type}
Theorem~\ref{theorem-associated-cubic} 
is of an essentially different nature than Theorem~\ref{theorem-associated-K3}, in that it does not ``come from'' K3 surfaces. 
More precisely, for a very general GM fourfold $X$ satisfying~\eqref{eq:GM:plane} for some $V_3$, the category $\cA_X$ is not equivalent 
to the derived category of a K3 surface, or even a twisted K3 surface. 
Indeed, the construction of Proposition~\ref{proposition-blowup-X} 
dominates the locus of cubic fourfolds containing a smooth cubic surface scroll, so 
it suffices to prove that given a very general such cubic, its K3 category is not 
equivalent to the twisted derived category of a K3 surface. 
Since cubic fourfolds containing a cubic scroll have discriminant $12$ by \cite[Example 4.1.2]{hassett}, 
this follows from~\cite[Theorem 1.4]{huybrechts-cubic}.
\end{remark}

\subsection{Strategy of the proof}
From now on, we assume the hypothesis of  
Theorem~\ref{theorem-associated-cubic} is satisfied. 
The proof of this theorem occupies the rest of this section. 
Here is our strategy.

By Orlov's decomposition of the derived category of a blowup, we have 
\begin{equation*}
\Db(\wtilde{X}) = \langle p^*\Db(X), i_*p_E^*\Db(P) \rangle.
\end{equation*}
Inserting~\eqref{eq:dbx-detailed} and the standard decomposition of $\Db(P)$ into the above decomposition, we obtain
\begin{equation}
\label{sod-with-plane-1} 
\Db(\wtilde{X}) = \langle p^*\cA_X, \cO, \cU^\vee, \cO(H), \cU^\vee(H),  
\cO_E, \cO_E(H), \cO_E(2H) \rangle.
\end{equation}
Here and below, to ease notation we write $\cU^\vee$ for $p^*\cU^\vee_X$.
This decomposition of $\Db(\wtilde{X})$ consists of a copy of $\cA_X$ and 
$7$ exceptional objects.

On the other hand, from the expression of $\wtilde{X}$ as a blowup of 
$X'$, we have
\begin{equation*}
\Db(\wtilde{X}) = \langle q^*\Db(X'), j_*q_{E'}^*\Db(T) \rangle.
\end{equation*} 
Inserting the decomposition~\eqref{sod-cubic} for $\Db(X')$, we obtain 
\begin{equation}
\label{sod-with-plane-2} 
\Db(\wtilde{X}) = \langle q^*\cA_{X'}, \cO, \cO(H'), \cO(2H'), j_*q_{E'}^*\Db(T) \rangle.
\end{equation} 
Note that $\Db(T)$ has a decomposition consisting of $4$ 
exceptional objects, hence the decomposition~\eqref{sod-with-plane-2} consists 
of one copy of $\cA_{X'}$ and again $7$ exceptional objects. 

To prove the equivalence $\cA_{X} \simeq \cA_{X'}$, we will find a sequence of mutations 
transforming the exceptional objects of~\eqref{sod-with-plane-1} into those 
of~\eqref{sod-with-plane-2}. In doing so, we will explicitly identify a 
functor giving the desired equivalence, see~\eqref{simpler-associated-cubic-equivalence}.

\subsection{Mutations} 

We perform a sequence of mutations, 
starting with~\eqref{sod-with-plane-1}. 
For a brief review of mutation functors and references, see the discussion in~\S\ref{subsection:self-duality}. 

\medskip \noindent
\textbf{Step 1.} Mutate $\cU^\vee(H)$ to the far left in~\eqref{sod-with-plane-1}. 
Since this is a mutation in $\Db(X)$ and we have~$K_X = -2H$, by~\eqref{eq:mutation-serre-functor} the result is
\begin{equation*}
\Db(\wtilde{X}) = \langle \cU^\vee(-H), p^*\cA_X, \cO, \cU^\vee, \cO(H),  
\cO_E, \cO_E(H), \cO_E(2H) \rangle.
\end{equation*}

\medskip \noindent
\textbf{Step 2.} Mutate $\cU^\vee(-H)$ to the far right. Again by~\eqref{eq:mutation-serre-functor} 
and~\eqref{eq:canclass:tx}, the result is 
\begin{equation*}
\Db(\wtilde{X}) = \langle p^*\cA_X, \cO, \cU^\vee, \cO(H),  
\cO_E, \cO_E(H), \cO_E(2H), \cU^\vee(H-E)  \rangle.
\end{equation*}

\medskip \noindent
\textbf{Step 3.} Left mutate $\cO_E$ through $\langle \cO, \cU^\vee, \cO(H) \rangle$. 
We have
\begin{align*}
&\Ext^\bullet(\cO(H),\cO_E)  = \rH^\bullet(P,\cO_P(-H)) = 0, \\
&\Ext^\bullet(\cU^\vee, \cO_E)  = \rH^\bullet(P, \cU_P) = 0, \\
& \Ext^\bullet(\cO, \cO_E)  = \rH^\bullet(P, \cO_P) = \bC,
\end{align*}
where in the second line $\cU_P$ is the tautological rank $2$ bundle on 
$P = \G(2,V_3)$, i.e. the restriction of $\cU$ from $\bG$ to $P$.
Hence by the definition of the mutation functor
\begin{equation*}
\rL_{\langle \cO, \cU^\vee, \cO(H) \rangle}(\cO_E) = \Cone(\cO \to \cO_E) = \cO(-E)[1],
\end{equation*}
and the resulting decomposition is
\begin{equation*}
\Db(\wtilde{X}) = \langle p^*\cA_X, \cO(-E), \cO, \cU^\vee, \cO(H), 
\cO_E(H), \cO_E(2H), \cU^\vee(H-E)  \rangle.
\end{equation*}

\medskip \noindent
\textbf{Step 4.} Left mutate $\cO_E(2H)$ through $\langle \cO, \cU^\vee, \cO(H), \cO_E(H) \rangle$. 

\begin{lemma}
We have $\rL_{\langle \cO, \cU^\vee, \cO(H), \cO_E(H) \rangle}(\cO_E(2H)) \cong \cO_{E'}(E'-F)[2]$.
\end{lemma}

\begin{proof}
There is an isomorphism of functors 
\begin{equation*}
\rL_{\langle \cO, \cU^\vee, \cO(H), \cO_E(H) \rangle} 
\cong 
\rL_{\cO} \circ \rL_{\cU^{\vee}} \circ \rL_{\cO(H)} \circ \rL_{\cO_E(H)}. 
\end{equation*}
Hence to prove the result we successively left mutate $\cO_E(2H)$ through $\cO_E(H), \cO(H), \cU^\vee, \cO$.

To compute $\rL_{\cO_E(H)}(\cO_E(2H))$, we may compute $\rL_{\cO_P(H)}(\cO_P(2H))$ 
and pull back the result. We have $\Ext^\bullet(\cO_P(H), \cO_P(2H)) = \rH^\bullet(P, \cO_P(H)) = V_3$, 
so 
\begin{equation*}
\rL_{\cO_P(H)}(\cO_P(2H)) = \Cone( \cO_P(H) \otimes V_3 \to \cO_P(2H) ).
\end{equation*}
The morphism $\cO_P(H) \otimes V_3 \to \cO_P(2H)$ is the twist by $H$ of the tautological 
morphism, hence it is surjective with kernel $\cU_P(H) \cong \cU^\vee_P$. Thus the above 
cone is $\cU_P^\vee[1]$, and 
\begin{equation*}
\label{step-4-1}
\rL_{\cO_E(H)}(\cO_E(2H)) = \cU_E^\vee[1].
\end{equation*}
Next note $\Ext^\bullet(\cO(H),\cU^\vee_E) = \rH^\bullet(P,\cU^\vee_P(-H)) = 0$, hence 
\begin{equation*}
\label{step-4-2}
\rL_{\cO(H)}(\cU_E^\vee) = \cU_E^\vee.
\end{equation*}
Further, we have $\Ext^\bullet(\cU^\vee,\cU^\vee_E) = \rH^\bullet(P, \cU_P\otimes \cU^\vee_P) = \bC$, hence
\begin{equation*}
\rL_{\cU^\vee}(\cU^\vee_E) = \Cone(\cU^\vee \to \cU^\vee_E) = \cU^\vee(-E)[1].
\end{equation*}
Now we are left with the last and most interesting step --- the mutation of $\cU^\vee(-E)$ through~$\cO$.  
First, using the exact sequence
\begin{equation*}
0 \to \cO(-E) \to \cO \to \cO_E \to 0
\end{equation*}
tensored by $\cU^\vee$, we find
\begin{equation}
\label{cohomology-u-e}
\Ext^\bullet(\cO,\cU^\vee(-E)) = \rH^\bullet(\wtilde{X},\cU^\vee(-E)) = \ker(V_5^\vee \to V_3^\vee) = V_2^\vee.
\end{equation}
Thus we need to understand the cone of the natural morphism $V_2^\vee \otimes \cO \to \cU^\vee(-E)$. 
Restricting~\eqref{us} to $\wtilde{X}$, dualizing, twisting by $H' = H-E$, and using 
the isomorphism $\cU(H) \cong \cU^\vee$, we obtain a distinguished triangle
\begin{equation*}
\label{keyseq}
V_2^\vee \otimes \cO \to \cU^\vee(-E) \to \cO_{E'}(E'-F).
\end{equation*} 
Thus 
\begin{equation}
\label{lousme}
\rL_\cO(\cU^\vee(-E)) = \cO_{E'}(E'-F),
\end{equation} 
which completes the proof of the lemma.
\end{proof}

By the lemma, the result of the above mutation is
\begin{equation*}
\Db(\wtilde{X}) = \langle p^*\cA_X, \cO(-E), \cO_{E'}(E'-F), \cO, \cU^\vee, \cO(H), 
\cO_E(H), \cU^\vee(H-E)  \rangle.
\end{equation*}

\medskip \noindent
\textbf{Step 5.} Left mutate $\cO_E(H)$ through $\cO(H)$. We have
\begin{equation*}
\rL_{\cO(H)}(\cO_E(H)) = \Cone(\cO(H) \to \cO_E(H)) = \cO(H-E)[1] = \cO(H')[1],
\end{equation*}
so the result is
\begin{equation*}
\Db(\wtilde{X}) = \langle p^*\cA_X, \cO(-E), \cO_{E'}(E'-F), \cO, \cU^\vee, \cO(H'),
\cO(H), \cU^\vee(H-E)  \rangle.
\end{equation*}

\medskip \noindent
\textbf{Step 6.} Right mutate $\cU^\vee$ through $\cO(H')$. We have 
\begin{equation*}
\Ext^\bullet(\cU^\vee,\cO(H')) = \Ext^\bullet(\cO,\cU(H-E)) = \Ext^\bullet(\cO,\cU^\vee(-E)) = V_2^\vee,
\end{equation*}
where the last equality holds by~\eqref{cohomology-u-e}. Hence
\begin{equation*}
\rR_{\cO(H')}(\cU^\vee) = \Cone(\cU^\vee \to V_2 \otimes \cO(H'))[-1].
\end{equation*}
Now restricting~\eqref{us} to $\wtilde{X}$ shows 
$\rR_{\cO(H')}(\cU^\vee) = \cO_{E'}(H' + F)[-1]$. Thus under the above mutation 
our decomposition becomes
\begin{equation*}
\Db(\wtilde{X}) = \langle p^*\cA_X, \cO(-E), \cO_{E'}(E'-F), \cO, \cO(H'), \cO_{E'}(H' + F),
\cO(H), \cU^\vee(H-E)  \rangle.
\end{equation*}

\medskip \noindent
\textbf{Step 7.} Left mutate $\cU^\vee(H-E)$ through $\cO(H)$. 
By~\eqref{lousme} and~\eqref{eq:mutation-equivalence} we have
\begin{equation*}
\rL_{\cO(H)}(\cU^\vee(H-E)) = \cO_{E'}(H + E' - F) = \cO_{E'}(2H' - F), 
\end{equation*}
so the result is
\begin{equation*}
\Db(\wtilde{X}) = \langle p^*\cA_X, \cO(-E), \cO_{E'}(E'-F), \cO, \cO(H'), \cO_{E'}(H' + F),
\cO_{E'}(2H' - F), \cO(H) \rangle.
\end{equation*} 

\medskip \noindent
\textbf{Step 8.} Right mutate $p^*\cA_X$ through $\langle \cO(-E), \cO_{E'}(E'-F) \rangle$. 
The result is
\begin{equation*}
\Db(\wtilde{X}) = \langle \cO(-E), \cO_{E'}(E'-F), \Psi p^*\cA_X, \cO, \cO(H'), \cO_{E'}(H' + F),
\cO_{E'}(2H' - F), \cO(H) \rangle,
\end{equation*} 
where $\Psi = \rR_{\langle \cO(-E), \cO_{E'}(E'-F) \rangle}$.

\medskip \noindent
\textbf{Step 9.} Mutate $\langle \cO(-E), \cO_{E'}(E'-F) \rangle$ to the far right. 
By~\eqref{eq:mutation-serre-functor}, the result is
\begin{equation*}
\Db(\wtilde{X}) = \langle \Psi p^*\cA_X, \cO, \cO(H'), \cO_{E'}(H' + F),
\cO_{E'}(2H' - F), \cO(H) , \cO(2H'), \cO_{E'}(3H'-F) \rangle.
\end{equation*} 

\medskip \noindent
\textbf{Step 10.}
Right mutate $\cO(H)$ through $\cO(2H')$. 
We have
\begin{equation*}
\Ext^\bullet(\cO(H),\cO(2H')) = \rH^\bullet(\wtilde{X}, \cO(E')) = \bC
\end{equation*}
and hence
\begin{equation*}
\rR_{\cO(2H')}(\cO(H)) = \Cone(\cO(H) \to \cO(2H'))[-1].
\end{equation*}
The morphism $\cO(H) \to \cO(2H')$ is the twist by $2H'$ of 
$\cO(-E') \to \cO$, hence 
\begin{equation*}
\rR_{\cO(2H')}(\cO(H)) = \cO_{E'}(2H')[-1].
\end{equation*}
Thus the result of the mutation is a decomposition
\begin{equation*}
\Db(\wtilde{X}) = \langle \Psi p^*\cA_X, \cO, \cO(H'), \cO_{E'}(H' + F),
\cO_{E'}(2H' - F), \cO(2H'), \cO_{E'}(2H'), \cO_{E'}(3H'-F) \rangle.
\end{equation*} 

\medskip \noindent
\textbf{Step 11.} 
Left mutate $\cO(2H')$ through $\langle \cO_{E'}(H' + F), \cO_{E'}(2H' - F) \rangle$. 
By the semiorthogonality of $q^*\Db(X')$ and $j_*q^*_{E'}\Db(T)$ in $\Db(\wtilde{X})$, 
this mutation is just a transposition. Thus the result is
\begin{equation*}
\Db(\wtilde{X}) = \langle \Psi p^*\cA_X, \cO, \cO(H'), \cO(2H'), \cO_{E'}(H' + F),
\cO_{E'}(2H' - F), \cO_{E'}(2H'), \cO_{E'}(3H'-F) \rangle.
\end{equation*} 
It is straightforward to check that
\begin{equation*}
\Db(T) = \langle  \cO_{T}(H' + F), \cO_{T}(2H' - F), \cO_{T}(2H'), \cO_{T}(3H'-F) \rangle,
\end{equation*}
so the above decomposition can be written as
\begin{equation}
\label{sod-with-plane-final}
\Db(\wtilde{X}) = \langle \Psi p^*\cA_X, \cO, \cO(H'), \cO(2H'), j_*q_{E'}^*\Db(T) \rangle.
\end{equation} 

This completes the proof of Theorem~\ref{theorem-associated-cubic}. 
Indeed, comparing the decompositions \eqref{sod-with-plane-final} and~\eqref{sod-with-plane-2} 
shows 
\begin{equation}
\label{associated-cubic-equivalence}
q_* \circ \rR_{\langle \cO_{\wtilde{X}}(-E), \cO_{E'}(E'-F) \rangle} \circ p^*\colon \cA_X \to \cA_{X'}
\end{equation}
is an equivalence. \qed 

\begin{remark}
The functor~\eqref{associated-cubic-equivalence} is in fact isomorphic to 
\begin{equation}
\label{simpler-associated-cubic-equivalence}
q_* \circ \rR_{\cO_{\wtilde{X}}(-E)} \circ p^*\colon \cA_X \to \cA_{X'}. 
\end{equation} 
To see this, observe that $q_*$ kills $\cO_{E'}(E'-F)$: 
if $j_0\colon T \hookrightarrow X'$ denotes the inclusion, then 
\begin{equation*}
q_*(\cO_{E'}(E'-F)) = j_{0*}q_{E'*}(\cO_{E'}(E'-F)) = j_{0*}(q_{E'*}(\cO_{E'}(E')) \otimes \cO_T(F)) = 0 
\end{equation*}
since $q_{E'*}(\cO_{E'}(E')) = 0$. 
Thus $q_* \circ \rR_{\cO_{E'}(E'-F)} \cong q_*$, and the claim follows 
since there is an isomorphism of functors 
$\rR_{\langle \cO_{\wtilde{X}}(-E), \cO_{E'}(E'-F) \rangle} \cong \rR_{\cO_{E'}(E'-F)} \circ \rR_{\cO_{\wtilde{X}}(-E)}$.
\end{remark}

\appendix

\section{Moduli of GM varieties}
\label{appendix-moduli}

Let $(\mathrm{Sch}/\bC)$ denote the category of $\bC$-schemes.
\begin{definition}
For $2 \leq n \leq 6$, the moduli stack $\cM_n$ of smooth $n$-dimensional GM varieties is the fibered category over $(\mathrm{Sch}/\bC)$ 
whose fiber over $S \in (\mathrm{Sch}/\bC)$ is the groupoid of pairs \mbox{$(\pi \colon X \to S, \cL)$}, 
where $\pi \colon X \to S$ is a smooth proper morphism of schemes and $\cL \in \Pic_{X/S}(S)$, 
such that for every geometric point $\bar{s} \in S$ the pair $(X_{\bar{s}}, \cL_{\bar{s}})$ is isomorphic 
to a smooth $n$-dimensional GM variety with its natural polarization (equivalently, $(X_{\bar{s}}, \cL_{\bar{s}})$ satisfies 
conditions~\eqref{eq:pic-x} and~\eqref{eq:omega-x} with $H$ the divisor corresponding to $\cL_{\bar{s}}$). 
A morphism from $(\pi' \colon X' \to S', \cL')$ to $(\pi \colon X \to S, \cL)$ is a 
fiber product diagram 
\begin{equation*}
\xymatrix{
X' \ar[d]_{\pi'} \ar[r]^{g'} & X \ar[d]^\pi \\ 
S' \ar[r]^g & S
}
\end{equation*}
such that $(g')^*(\cL) = \cL' \in \Pic_{X'/S'}(S')$.
\end{definition}

The following result gives the basic properties of the moduli stack $\cM_n$. 
An explicit description of $\cM_n$ will be given in~\cite{debarre-kuznetsov-moduli}. 
We follow~\cite{stacks-project} for our conventions on algebraic stacks. 

\begin{proposition}
\label{proposition-moduli}
The moduli stack $\cM_{n}$ is a smooth and irreducible Deligne--Mumford stack of finite type over $\bC$. 
Its dimension is given by $\dim \cM_n = 25 - (6-n)(5-n)/2$, i.e. 
\begin{equation*}
\dim \cM_2 = 19, ~ \dim \cM_{3} = 22, ~ \dim \cM_{4} = 24, ~ \dim \cM_{5} = 25, ~ \dim \cM_{6} = 25.
\end{equation*}
\end{proposition}

We will use the following lemma. 
\begin{lemma}
\label{lemma-moduli}
Let $X$ be a smooth GM variety of dimension $n \geq 3$. Then: 
\begin{enumerate}
\item \label{lemma-moduli-1} The automorphism group scheme $\Aut_{\bC}(X)$ is finite and reduced. 
\item \label{lemma-moduli-2} $\rH^i(X, \rT_X) = 0$ for $i \neq 1$. 
\item \label{lemma-moduli-3} $\dim \rH^1(X, \rT_X) = 25 - (6-n)(5-n)/2$. 
\end{enumerate}
\end{lemma} 

\begin{proof}
As our base field $\bC$ has characteristic $0$, 
$\Aut_{\bC}(X)$ is automatically reduced by a theorem of Cartier~\cite[Lecture~25]{mumford},
and it is finite by~\cite[Proposition 3.21(c)]{debarre-kuznetsov}. 
Hence $\rH^0(X, \rT_X)$, being the tangent space to $\Aut_{\bC}(X)$ at the identity, vanishes. 
Further, $\rT_X \cong \Omega_X^{n-1}(n-2)$ by~\eqref{eq:omega-x}
and hence $\rH^i(X, \rT_X) = 0$ for $i \geq 2$ by Kodaira--Akizuki--Nakano vanishing. 
Finally, the dimension of $\rH^1(X,\rT_X)$ is straightforward to compute using 
Riemann--Roch. 
\end{proof}

\begin{proof}[Proof of Proposition~\textup{\ref{proposition-moduli}}]
First consider the case $n=2$. 
Then by~\eqref{eq:pic-x},
$\cM_2$ is the Brill--Noether 
general locus (and hence Zariski open) in the moduli stack of polarized K3 surfaces of degree~$10$.
It is well-known that all the properties in the proposition hold for the moduli 
stack of primitively polarized K3 surfaces of a fixed degree (see~\cite[Chapter 5]{huybrechts-K3}), 
so they also hold for~$\cM_{2}$. 

From now on assume $n \geq 3$. 
A standard Hilbert scheme argument shows that $\cM_n$ is an algebraic 
stack of finite type over $\bC$, whose diagonal is affine and of finite type.
To prove $\cM_n$ is \mbox{Deligne--Mumford}, by~\cite[Tag 06N3]{stacks-project} 
it suffices to show its diagonal is unramified. 
As a finite type morphism is unramified if and only if all of its geometric fibers 
are finite and reduced, we are done by Lemma~\ref{lemma-moduli}\eqref{lemma-moduli-1} 
(note that for a GM variety of dimension $n \geq 3$, all automorphisms preserve the natural polarization). 

Next we check smoothness of $\cM_n$. Let $(X, \cL)$ be a point of $\cM_n$, i.e. $X$ is a 
GM $n$-fold and $\cL \in \Pic(X)$ is the ample generator. 
Let $\fA_{\cL}$ be the Atiyah extension of $\cL$, i.e. the extension
\begin{equation*}
0 \to \cO_X \to \fA_{\cL} \to \rT_X \to 0 
\end{equation*}
given by the Atiyah class of $\cL$.
Further, recall that $\rH^1(X, \fA_\cL)$ classifies first order deformations of the 
pair $(X, \cL)$, and $\rH^2(X, \fA_\cL)$ is the obstruction space for such deformations 
(see~\cite[\S3.3.3]{sernesi}). 
Taking cohomology in the above sequence shows that $\rH^i(X, \fA_{\cL}) \cong \rH^i(X, \rT_X)$ 
for $i \geq 1$. 
In particular, $\rH^2(X, \fA_{\cL}) = 0$ by Lemma~\ref{lemma-moduli}\eqref{lemma-moduli-2}, 
so the formal deformation space of $\cM_n$ at $(X,\cL)$ is smooth of dimension 
$\dim \rH^1(X, \fA_\cL) = \dim \rH^1(X, \rT_X)$. 
This implies the smoothness of $\cM_n$ and, using 
Lemma~\ref{lemma-moduli}\eqref{lemma-moduli-3}, the formula for its dimension. 

It remains to show that $\cM_n$ is irreducible. 
This follows from the defining expression~\eqref{x:gm:general} of any GM variety. 
Indeed, let $P_n$ be the space of pairs $(W, Q)$ where $W \subset \bC \oplus \wedge^2V_5$ is an $(n+5)$-dimensional linear 
subspace and $Q \subset \bP(W)$ is a quadric hypersurface, and let $U_n \subset P_n$ be the open subset where $\Cone(\bG) \cap Q$ is smooth of dimension $n$. 
The projection $P_n \to \G(n+5,\bC \oplus \wedge^2V_5)$ is a projective bundle, 
hence $P_n$ and $U_n$ are irreducible. 
On the other hand, by~\eqref{x:gm:general}, $U_n$ maps surjectively onto $\cM_n$. 
Hence $\cM_n$ is irreducible as well. 
\end{proof}


\providecommand{\bysame}{\leavevmode\hbox to3em{\hrulefill}\thinspace}
\providecommand{\MR}{\relax\ifhmode\unskip\space\fi MR }
\providecommand{\MRhref}[2]{%
  \href{http://www.ams.org/mathscinet-getitem?mr=#1}{#2}
}
\providecommand{\href}[2]{#2}

\end{document}